\documentclass{amsart}%
\usepackage{amsmath}
\usepackage{enumitem} 
\usepackage{color}
\usepackage{amssymb}
\usepackage{amsfonts}
\usepackage{graphicx}%

\newtheorem{theorem}{Theorem}[section]
\theoremstyle{plain}

\newtheorem{question}{Question}
\newtheorem{corollary}[theorem]{Corollary}

\newtheorem{example}[theorem]{Example}

\newtheorem{lemma}[theorem]{Lemma}

\newtheorem{proposition}[theorem]{Proposition}
\newtheorem{remark}[theorem]{Remark}

\numberwithin{equation}{section}

\def \R{\mathbb R}

\def \S{\mathbb S}

\def \N{\mathbb N}

\def \H{\mathbb H}
\def \S{\mathbb{S} }
\def \grad{\text{grad}}
\begin{document}
\title[An ODE reduction method for the semi-Riemannian
Yamabe problem.]{An ODE reduction method for the semi-Riemannian
	Yamabe problem on space forms.}

\author{Juan Carlos Fern\'{a}ndez}
\address{Departamento de Matem\'{a}ticas, Facultad de Ciencias, Universidad Nacional Aut\'{o}noma de M\'{e}xico, CP 04510, M\'{e}xico}
\email{jcfmor@ciencias.unam.mx}
\author{Oscar Palmas}
\address{Departamento de Matem\'{a}ticas, Facultad de Ciencias, Universidad Nacional Aut\'{o}noma de M\'{e}xico, CP 04510, M\'{e}xico}
\email{oscar.palmas@ciencias.unam.mx}
\thanks{J.C. Fernández was supported by a postdoctoral fellowship from UNAM-DGAPA}
\thanks{O. Palmas was partially supported by UNAM under Project PAPIIT-DGAPA IN115119}
\date{\today}

\begin{abstract}
	We consider the semi-Riemannian Yamabe type equations of the form
	\[
	-\square u + \lambda u = \mu \vert u\vert^{p-1}u\quad\text{ on }M
	\]
where $M$ is either the semi-Euclidean space or the pseudosphere of dimension $m\geq 3$, $\square$ is the semi-Riemannian Laplacian in $M$, $\lambda\geq0$, $\mu\in\R\smallsetminus\{0\}$ and $p>1$. Using semi-Riemannian isoparametric functions on $M$, we reduce the PDE into a generalized Emden-Fowler ODE of the form
\[
w''+q(r)w'+\lambda w = \mu\vert w\vert^{p-1}w\quad\text{ on } I,
\]
where $I\subset\R$ is $[0,\infty)$ or $[0,\pi]$, $q(r)$ blows-up at $0$ and $w$ is subject to the natural initial conditions $w'(0)=0$ in the first case and $w'(0)=w'(\pi)=0$ in the second. We prove the existence of blowing-up and globally defined solutions to this problem, both positive and sign-changing, inducing solutions to the semi-Riemannian Yamabe type problem with the same qualitative properties, with level and critical sets described in terms of semi-Riemannian isoparametric hypersurfaces and focal varieties. In particular, we prove the existence of sign-changing blowing-up solutions to the semi-Riemannian Yamabe problem in the pseudosphere having a prescribed number of nodal domains.



Secondary:
35B08, 35B33, 35B44,  53C40.

\bigskip

\end{abstract}
\maketitle

\section{\textbf{Introduction}}

We say that two semi-Riemannian metrics $g$ and $\hat{g}$ on a manifold $M$ of dimension $m\geq 3$ are conformal if there exists a positive function $\rho\in\mathcal{C}^\infty(M)$ such that $\hat{g}=\rho g$. Given a 
semi-Riemannian manifold $(M,g)$ without boundary and signature $s\geq 0$, the semi-Riemannian Yamabe problem consists in finding a metric $\hat g$ conformal to $g$ in such a way that 
 the scalar curvature of $(M,\hat{g})$ is constant. If we consider a conformal metric $\hat{g}=u^{\frac{4}{m-2}} g$ and denote by $R_g$ and $R_{\hat{g}}$ the scalar curvatures with respect to $g$ and $\hat{g}$, respectively, then the following relation between them holds true (see Section 1.J in \cite{Be})
\[
R_{\hat{g}}=u^{-\frac{m+2}{m-2}}\left[-4\frac{m-1}{m-2}\square_{g}u + R_g u\right].
\]
where $\square_g=-\text{div}_g\nabla_g$ denotes the Laplacian on the semi-Riemannian manifold $(M,g)$. If we look for $R_{\hat{g}}$ to be constant, say equal to $\mu\in\R$, by the above relation, we obtain that the Yamabe problem is equivalent to finding a positive solution to the PDE
\begin{equation}\label{Eq:MainLorYam}
-\square_{g}u + \mathfrak{c}_mR_g u=\mu \vert u\vert^{p_m-1}\qquad \text{ on }\ M
\end{equation}
with critical Sobolev exponent $p_m:=\frac{m+2}{m-2}$ and where $\mathfrak{c}_m=\frac{m-2}{4(m-1)}$. 

When the signature of the metric $g$ is $s=0$, then the Laplacian is just the usual Laplace-Beltrami operator, $\square_g=\Delta_g$, and the semi-Riemannian Yamabe problem reduces to the usual Yamabe equation on complete Riemannian manifolds, which is an elliptic PDE widely studied in the last sixty years (see, for instance the book and surveys \cite{Au1,BrMa1,LePa} and the references therein). When the signature $s=1$, $(M,g)$ is a Lorentzian manifold, $\square_g$ is the D'Alembert operator and the nonlinear wave equation equation \eqref{Eq:MainLorYam} is referred as the Lorentzian or hyperbolic Yamabe equation \cite{Gi,KoLi}. For $s\geq 2$, the operator $\square_g$ is ultrahyperbolic. 

Hereafter, we will also denote the semi-Riemannian metric $g$ as $\langle\cdot,\cdot\rangle_g$ or simply by $\langle\cdot,\cdot\rangle$ if there is no risk of confusion. In what follows, we will denote $z\in M$ as $z=(\bar{t},\bar{x})$ in coordinates, where $\bar{t}=(t_1,\ldots,t_s)$ will be called \emph{time variables}, while $\bar{x}=(x_1,\ldots,x_{m-s})$ will be called the \emph{space variables}, or simply by $z=(z_1,\ldots,z_m)$ in convenience. 

In case of the Minkowski space $\R^{m}_1$, equation \eqref{Eq:MainLorYam} takes the more familiar form
\begin{equation}\label{Eq:CritEnWaveEq}
\frac{\partial^2 u}{\partial t^2} - \Delta u = \mu \vert u\vert^{p_m-1}u\qquad\text{ in }\ \R_1^1\times\R^{m-1}.
\end{equation}
This is the so called \emph{critical energy nonlinear wave equation} widely studied (see, for instance \cite{Ke,KeMe}  and the references therein). More generally, for the case of the semi-Euclidean space $\R_s^m\equiv\R^s_s\times\R^{m-s}$ with coordinates $(\bar{t},\bar{x})=(t_1,\ldots,t_s,x_1,\ldots,x_{m-s})$, the semi-Riemannian Yamabe problem can be written as an ultrahyperbolic nonlinear equation
\begin{equation}\label{Eq:UltraWaveEq}
\Delta_{\bar{t}} u - \Delta_{\bar{x}} u = \mu \vert u\vert^{p_m-1}u\qquad\text{ in }\R^s_s\times\R^{m-s}.
\end{equation}
We are going to study the more general Yamabe type problem in semi-Euclidean space form 
\begin{equation}\label{Eq:MainYamabeTypeMinkowski}
-\square u = \mu\vert u\vert^{p-1}u,\qquad \text{ in }\R_s^m
\end{equation}
with $m\geq 3$, $s\geq 1$, $p>1$ and $\mu\in\R\smallsetminus\{0\}$.  The case $\mu=0$ was already studied in \cite{KoOr}.

 Recall the definition of the  pseudosphere, also known as the de Sitter space when $s=1$,
\[
\S^m_s:=\{z\in\R^{m+1}_s\;:\;\langle z,z\rangle=1\}.
\]
with its canonical induced metric. As the scalar curvature of the pseudosphere is $m(m-1)$, the semi-Riemannian Yamabe equation on this manifold is
\begin{equation}\label{Eq:MainYamabeSphere}
-\square_{g}u + \lambda_m u=\mu \vert u\vert^{p_m-1}u\qquad \text{ on }\ \S_s^m
\end{equation}
where $\lambda_m:=\frac{m(m-2)}{4}>0$ and $\mu>0$. We will focus on the more general equation
\begin{equation}\label{Eq:MainYamabeSphereNorm}
-\square_{g}u + \lambda u=\lambda \vert u\vert^{p-1}u\qquad \text{ on }\ \S^m_s
\end{equation} 
where $\lambda>0$ is constant and $p>1$. When $\lambda=\lambda_m$ and $p=p_m$, this equation is a renormalization of  \eqref{Eq:MainYamabeSphere}, since $u$ is a solution to \eqref{Eq:MainYamabeSphereNorm}  if and only if $\hat{u}=\left(\frac{\lambda}{\mu}\right)^{\frac{1}{p-1}}u$ is a solution to \eqref{Eq:MainYamabeSphere} for any $\mu> 0$.

Here we will reduce equations \eqref{Eq:MainYamabeTypeMinkowski} and \eqref{Eq:MainYamabeSphereNorm} into singular nonlinear ODE's. The reduction uses the theory of \emph{isoparametric functions and hypersurfaces}, which we briefly recall in order to state our main results. We take the definitions and basic results from \cite{Ha}.

Given a semi-Riemannian space form $(M,g)$, a nondegenerate hypersurface $S$ of $M$ is called \emph{isoparametric} if the principal curvatures and their algebraic multiplicities are constant on $S$. The family of the hypersurfaces parallel to $S$ is called an \emph{isoparametric family}. The problem of classifying isoparametric hypersurfaces on Riemannian space forms was settled by Cartan, completely solved by Cartan himself in case of the Euclidean and hyperbolic spaces and partially solved in case of the round sphere \cite{Ca1,Ca2,Ca3}. The case of the sphere remains open but considerable progress has been made toward the classification of isoparametric hypersurfaces on this manifold, see for instance the book \cite{CeRy}, the articles \cite{Ch,Mi1} and the references therein. The classification problem in Lorentzian space forms began with Nomizu \cite{No} and continued with the works of Magid \cite{Ma} on the Minkowski space, Li \cite{Li} on the De Sitter space and Xiao \cite{Xi} on the anti De Sitter space. However, 
the classification is not complete, even in the Minkowski space (see the recent survey \cite{Na} and the references therein). The general problem of classifying isoparametric hypersurfaces in semi-Riemannian space forms is largely open and the main similarities and differences with the Riemannian case were pointed out by Hahn \cite{Ha}. The main difficulty here is that the shape operator is not necessarily diagonalizable, although the principal curvatures (real or complex) as well as their algebraic multiplicity and the minimal polynomial of the shape operator are everywhere constant. In case the shape operator has at most two distinct principal curvatures, the classification is available \cite{AbKoYa}; however, other possibilities may occur even in the semi-Riemannian Euclidean space (see \cite{Ha}).

A way to produce isoparametric hypersurfaces is by means of \emph{isoparametric functions}. A smooth function $\varphi:M\rightarrow\R$ is called \emph{isoparametric} if there exist smooth functions $a,b:\R\rightarrow\R$ such that
\begin{equation}\label{Eq:Isoparametric}
\square_g \varphi=a(\varphi)\quad\text{ and }\quad\langle\grad_g \varphi,\grad_g \varphi\rangle_g= b(\varphi).
\end{equation}
If $t\in\R$ is a regular value of $\varphi$, $S_t:=\varphi^{-1}(t)$ is an isoparametric hypersurface. On the other hand, if $S_0$ is isoparametric, then locally there exist hypersurfaces $S_t$, $t\in(-\varepsilon,\varepsilon)$, parallel to $S_0$, with constant mean curvature and the function defined as $\varphi(z)=t$ for $z\in S_t$ is isoparametric, but it may not be globally defined in $M$. In case of the Riemannian Euclidean space and the Riemannian round sphere, all isoparametric hypersurfaces are obtained as inverse images of regular values of homogeneous polynomials, they are all diffeomorphic, and the inverse image of the critical values are also smooth submanifolds, called \emph{focal submanifolds} (see \cite{CeRy,Mu}).  It is not clear  that these phenomena  happen for the semi-Riemannian space forms, in particular, that every isoparametric hypersurface is a leaf of a foliation given by a global isoparametric function. Even in the simplest case of Lorentzian space forms and Lorentzian isoparametric hypersurfaces, the existence of such a global isoparametric function is not clear for most of the interesting and nontrivial examples \cite{Ma,Xi,Li}. For instance, as it was mentioned in \cite{Ha}, the $B$-scroll hypersurfaces in $\R^m_1$ are not algebraic, hence if there is a global isoparametric function that defines it, it could not be a polynomial, contrasting with the aforementioned Riemannian setting \cite{CeRy}. Even worse, the existence of a global isoparametric function having the $B$-scroll hypersurfaces as level sets is not clear at all. 

In the rest of this work, we will focus on isoparametric hypersurfaces $S$ of a semi-Riemannian manifold $(M,g)$ for which a globally-defined isoparametric function $\varphi:M\rightarrow\R$ exists and is such that $S=f^{-1}(t)$, for some regular value $t\in$Im$\;\varphi$.

We provide some examples in the case of the semi-Euclidean space and the pseudosphere.

\begin{example}\label{Ex:IsoparametricMinkowsky} (Cf. \cite[Proposition 2.3]{Ha}.) 
 Let $A\in$Sym$(\R^m_s)$, $a\in\R^m_s$ and suppose the existence of a real number $\alpha\in\R$ such that $(A-\alpha I)A=0$ and $Aa=\alpha a$. Then $\varphi(z) =\langle Az, z \rangle + 2\langle a, z \rangle$ is isoparametric and
\begin{equation}\label{Eq:IsoparamFunctionMinkowski}
	\langle \emph{\text{grad}} \varphi,\emph{\text{grad}} \varphi\rangle=4\alpha \varphi+4\langle a,a \rangle\quad\text{and}\quad\square_g\varphi=2\emph{\text{tr}}A.
\end{equation} 
\end{example}

The level sets of isoparametric functions given by critical values will be called \emph{focal varieties}. For instance, in the above example, $V:=\varphi^{-1}(0)$ is a focal variety when $\alpha\neq0$. In Appendix \ref{App:Minkowski} we will give several examples of these isoparametric functions and we will describe their corresponding hypersurfaces and focal varieties.

\begin{remark}
When $s\geq1$, the focal varieties may not be even topological manifolds, but algebraic varieties, and they could separate two or more classes of non homeomorphic  isoparametric hypersurfaces generated by the same isoparametric function. For instance, when $s=1$, the function $\varphi(z)=\langle z,z \rangle$ is isoparametric, and its isoparametric hypersurfaces are homothetic either to the de Sitter space $\S_1^m$ or to the anti de Sitter space $\H^{m-1}:=\{z\in\R_s^{m}\;:\; \langle z,z \rangle = -1\}$, which are not homeomorphic, In this example, the focal variety is the light cone $\varphi^{-1}(0)=\mathcal{C}_1^{m-1}:=\{z\in\R_s^{m}\;:\; \langle z,z \rangle = 0\}$, which is an algebraic variety that is not a topological manifold.
\end{remark}

We now turn our attention to the pseudosphere.

\begin{example}\label{Ex:IsoparametricDeSitter}
 In case of the De Sitter space we have the following examples of global isoparametric hypersurfaces (see \cite{Ha}):
	\begin{enumerate}[leftmargin=*]
		\item \emph{Linear examples:} Given $Q\in\S_s^m$, the function $\varphi(z)=\langle z,Q \rangle$ is isoparametric with
		\begin{equation*}
		\langle \emph{\grad}_g \varphi, \emph{\grad}_g \varphi\rangle_g =(1-\varphi^2)\quad\text{ and }\quad \square_g \varphi=-m\varphi.
		\end{equation*} 
		\item \emph{Quadratic examples:} Let $A\in$Sym$(\R_s^{m+1})$, then the function $\varphi:\S_s^m 
 		\rightarrow\R$ given by $\varphi(z):=\langle Az,z \rangle$ is isoparametric if and only if the minimal polynomial of $A$ is
 		 $p_A(t)=t^2+\alpha t +\beta$ for some $\alpha,\beta\in\R$. In this case,
 		 \begin{equation*}
 		 \langle \emph{\grad}_g \varphi, \emph{\grad}_g \varphi\rangle_g =-4p_A(\varphi)\quad\text{ and }\quad \square_g \varphi=2\emph{\text{tr}}A-2(n+2)\varphi.
 		 \end{equation*}
 		 \item \emph{Clifford examples: } If $m=2k-1$, we say that the $n$-tuple $(P_1,\ldots,P_n)$ is a Clifford system of signature $(n,r)$ if $P_i\in$Sym$(\R_s^{m+1})$ and $P_iP_j+P_jP_i=2\eta_{ij}$ for $i,j=1,\ldots,n$, where $\eta_{ij}=-1$ if $i=j\leq r$, $\eta_{ij}=1$ if $r<i=j\leq n$ and $\eta_{ij}=0$ otherwise. In this setting, the function $\varphi:\S_s^m\rightarrow\R$ given by $\varphi(z):=\langle z,z \rangle^2 - 2\sum_{j=1}^{n}\eta_{jj}\langle P_jz,z \rangle^2$, is isoparametric with 
 		 \begin{equation*}
 		 \langle \emph{\grad}_g \varphi, \emph{\grad}_g \varphi\rangle_g =16(1-\varphi^2)\ \text{ and }\  \square_g \varphi=-4(m+3)\varphi+8(k+1).
 		 \end{equation*} 
 		 \end{enumerate}
\end{example}

In Appendix \ref{App:DeSitter} we give a better description of these isoparametric functions and their hypersurfaces, while in Section \ref{Sec:Redcution Method} we will see that all these examples fit into the general setting of semi-Riemannian Cartan-M\"{u}nzner polynomials, which are a generalization of the Riemannian ones  on round spheres \cite{CeRy,Mu}.

\bigskip

For semilinear PDE's defined in semi-Riemannian manifolds for which a non constant globally defined isoparametric function exists, we have the following reduction method.

\begin{proposition}\label{Prop:ReductionODE}
	Let $(M,g)$ be a semi-Riemannian space form, $\varphi:M\rightarrow\R$ be an isoparametric function with $a$ and $b$ satisfying \eqref{Eq:Isoparametric}, and $\Psi:\R\rightarrow\R$ a continuous function. Then $u:=v\circ \varphi$ satisfies equation
	\begin{equation}\label{Eq:GeneralNonlinearEDP}
	-\square_g u = \Psi(u)\qquad\text{ on }\ M
	\end{equation}
	if and only if $v:\text{Im}\,\varphi\rightarrow\R$ is a solution to
	\begin{equation}\label{Eq:GeneralNonlinearODE}
	-b(t)v''-a(t) v' = \Psi(v)\qquad\text{ on }\ \text{Im}\,\varphi
	\end{equation}
\end{proposition}

The proof follows immediately from the identity
\[
\square_g (v\circ \varphi)=(v''\circ \varphi)\langle \grad_g \varphi,\grad_g \varphi\rangle_g + (v'\circ \varphi)\;\square_g \varphi.
\]
 
We next proceed to state our main results, which will be a consequence of the above Proposition and the qualitative behavior of the solutions to the reduced ODE \eqref{Eq:GeneralNonlinearODE}. To this end, first we introduce some concepts and notation. If $u:M\rightarrow\R$ is smooth, the nodal set of $u$ is $\mathcal{N}:=\{u=0\}$, while its critical set is $\mathcal{Z}:=\{\nabla u = 0 \}$. Each of the connected components of $M\smallsetminus\mathcal{N}$ is called a nodal domain. Observe that $u$ can not change sign in a nodal domain, so that it is either positive or negative. 

For fixed positive integers $m\geq 3$, $1\leq s< m$, $0\leq k\leq s$ and $0\leq n\leq m-s$, and for $z=(\bar{t},\bar{x})\in\R_s^m$, define $\ell_{k,n}(z):=-\sum_{i=1}^k t_i^2 + \sum_{j=1}^nx_j^2$. Notice that $\ell_{k,0}(z)=-\vert(t_1,\ldots,t_k)\vert^2$, $\ell_{0,n}(z)=\vert(x_1,\ldots,x_n)\vert^2$ and if $k+n=m$, $\ell_{k,n}(z)=\langle z,z \rangle.$

Our first main theorem gives the existence of global and blowing-up solutions to the Yamabe problem in the semi-Riemannian Euclidean space form \eqref{Eq:MainYamabeTypeMinkowski}.

\begin{theorem}\label{Th:MinkowskiIsop}
Let $p>1$ and $\mu\ne 0$. Then equation \eqref{Eq:MainYamabeTypeMinkowski} admits the following bounded solutions globally defined in $\R_s^m$:
\begin{itemize}
	\item[(Y1)] For $k+n=0,1$ and any $p>1$, the solution $u$ is sign changing with an infinite number of nodal domains.
	\item[(Y2)] For $k+n\geq 2$, with $k=0$ or $n=0$ and $1<p<\frac{(k+n)+2}{(k+n)-2}\leq\infty$, $u$ is sign changing, has an infinite number of nodal domains and $\vert u(z)\vert\rightarrow 0$ as $\ell_{0,n}(z)\rightarrow\infty$ when $k=0$, or as $\ell_{k,0}(z)\rightarrow-\infty$ when $n=0$.
	\item[(Y3)] For $k+n>2$, with  $k=0$ or $n=0$ and $p\geq \frac{(k+n)+2}{(k+n)-2}$, $u$ is positive, with a single critical point at the origin, which is a global maximum, and $\vert u(z)\vert\rightarrow 0$ as $\ell_{0,n}(z)\rightarrow\infty$ when $k=0$, or as $\ell_{k,0}(z)\rightarrow-\infty$ when $n=0$.
\end{itemize}
Moreover, equation \eqref{Eq:MainYamabeTypeMinkowski} admits the following types of blowing-up solutions:
	\begin{itemize}
		\item[(Y4)] For $k+n\geq0$, with $k=0$ or $n=0$ and any $p>1$,  $u$ is positive, with a single critical point at the origin, which is a global minimum, and there exists $R>0$, $a\in\R^m_s\smallsetminus\{0\}$, $\epsilon_1,\ldots,\epsilon_s\in\{-1,1\}$ and $1\leq \omega\leq m-2s$ such that $\vert u(z)\vert\rightarrow\infty$ as either $\ell_{0,n}(z)\rightarrow R$ when $k=0$, as $\ell_{k,0}(z)\rightarrow-R$ if $n=0$, as $\langle a,z \rangle\rightarrow \pm R$ or as $\sum_{i=1}^s\vert t_i + \epsilon_i x_i\vert^2+2\sum_{j=1}^{s+\omega}x_j\rightarrow\pm R$ when $k=n=0$. 
		\item[(Y5)] For $k+n\geq 2$, with $k,n\neq0$ and $1<p<\frac{(k+n)+2}{(k+n)-2}\leq\infty$, $u$ is sign-changing, with an infinite number of nodal domains, $\vert u(z)\vert\rightarrow 0$ as $\ell_{k,n}(z)\rightarrow\infty$ (or as $\ell_{k,n}(z)\rightarrow -\infty$) and there exist $R>0$ such that $\vert u(z)\vert\rightarrow\infty$ as $\ell_{k,n}(z)\rightarrow-R$ (or as $\ell_{k,n}(z)\rightarrow R$).
		\item[(Y6)]  For $k+n>2$, with $k,n\neq0$ and $p\geq \frac{(k+n)+2}{(k+n)-2}$, $u$ is positive with a single critical point at the origin, which is a saddle point, $\vert u(z)\vert\rightarrow 0$ as $\ell_{k,n}(z)\rightarrow\infty$ (or as $\ell_{k,n}(z)\rightarrow -\infty$) and there exist $R>0$ such that $\vert u(z)\vert\rightarrow\infty$ as $\ell_{k,n}(z)\rightarrow-R$ (or as $\ell_{k,n}(z)\rightarrow R$).
	\end{itemize}
\end{theorem}

We can extract even more information about the underlying geometry of the solutions. The next result will describe the level sets, including $\mathcal{N}$, and the critical set of the solutions given in Theorem \ref{Th:MinkowskiIsop} in terms of isoparametric hypersurfaces and focal varieties in $\R_s^m$. In what follows, for $a\in\R_s^m\smallsetminus\{0\}$, $\mathcal{L}_a:=\{z\in\R_s^m\;:\;\langle a,z \rangle = 0\}$ will denote a hyperplane through the origin, $\mathcal{P}$ will denote one of the parabolic cylinders described in Appendix \ref{App:Minkowski}, while $\S_\nu^{N-1}(r):=\{ z\in\R_\nu^N \;:\; \langle z,z \rangle = -r \}$, $\H_{\nu-1}^{N-1}(r):=\{ z\in\R_\nu^N \;:\; \langle z,z \rangle = -r \}$ and $\mathcal{C}_\nu^{N-1}:=\{  z\in\R_\nu^N \;:\; \langle z,z \rangle = 0 \}$ will denote the pseudosphere of radius $r$, the pseudo-hyperbolic space of radius $r$ and the null cone in $\R_\nu^N$, respectively, where for $\nu=0$, $\S_0^{N-1}(r)=\S^{N-1}(r)$ is the standard sphere of radius $r$ in $\R_0^N=\R^N$.

\begin{theorem}\label{Th:GeometrySolutions} Let $u$ be one of the solutions (Y1)-(Y6) of the previous Theorem and let $S$ and $V$ be the connected components of a regular level set, including the nodal set $\mathcal{N}$, and the critical set $\mathcal{Z}$ of $u$, respectively. Then, up to isometries in $\R_s^m$:
\begin{itemize}[leftmargin=*]
\item For (Y1) and (Y4), $S$ is homothetic to either $\mathcal{L}_a^{m-1}$, $\R_s^{m-1}$, $\R^{m-1}_{s-1}$, $\mathcal{P}\times\R^{m-2s-\omega}$, $\R_s^{m-n}\times\S^{n-1}(1)$ or $\S^{k-1}(1)\times\R^{m-k}_{s-k}$, and only one option is possible, while $V$ is homothetic to either one of these hypersurfaces or to either one of the focal varieties $\R^{m-k}_{s-k}$ or $\R^{m-n}_{s}$. Moreover, the blow-up for (Y4) occur on a hypersurface homothetic to $S$.
\item For (Y2) and (Y3), $S$ is homothetic to either $\R^{m-1}_s$, $\R_{s-1}^{m-1}$, $\R_s^{m-n}\times\S^{n-1}(1)$ or $\S^{k-1}(1)\times\R^{m-k}_{s-k}$ and only one option, while $V$ is homothetic  to either one of these hypersurfaces or to  one of the focal varieties $\R^{m-k}_{s-k}$ or $\R^{m-n}_{s}$.
\item For (Y5) and (Y6), $S$ and $V$ are homothetic to either   $\S_k^{(k+n)-1}(1)\times \R^{m-(k+n)}_{s-k}$ or $\H_{k-1}^{(k+n)-1}(1)\times\R_{s-k}^{m-(k+n)}$, where both options occur, and $V$ could also be homothetic to the algebraic variety $\mathcal{C}_k^{(k+n)-1}\times\R_{s-k}^{m-(k+n)}$. Moreover, the blow-up occur at a hypersurface homothetic to $S$.
\end{itemize} 
\end{theorem}

When considering the critical exponent $p_m=\frac{m+2}{m-2}$, we observe that if either $k=0$ or $n=0$, then necessarily $p<\frac{(k+n)+2}{(k+n)-2}$. This leads us to the following immediate corollary.

\begin{corollary}
The Yamabe equation \eqref{Eq:MainLorYam} in $\R^{m}_s$ and $\mu\neq0$, admits all the solutions of Theorem \ref{Th:MinkowskiIsop} but (Y3).
\end{corollary}

We make the following remarks

\begin{enumerate}[leftmargin=*]
	\item The solutions (Y1) and (Y4) with hyperplanes as level sets are not new and these solutions arise, for instance, when considering $u$ depending only on one variable for the Lorentizan Yamabe problem. For example, Kong and Liu \cite{KoLi} and Ginoux \cite{Gi} studied this reduction taking $u(\bar{t},\bar{x})=u(t_1)$ in case $s=1$, i.e., fixing the space variable, giving a complete description of the solutions. Our result generalize theirs for every $s\geq1$, for their result can be recovered from ours by taking $a=(1,0,0,\ldots,0)$ and, in this case, the blow up occur as $t_1\rightarrow\pm R$.
	\item For the critical exponent $p_m$, signature $s=1$ and taking $k+n=m$, the solutions (Y5) and (Y6) are the radial solutions to the (focusing or defocusing) critical energy wave equation \eqref{Eq:CritEnWaveEq} (see \cite{Ke} or \cite{KeMe}). For any signature $s\geq 1$, the solutions (Y5) and (Y6) with $k+n=m$, are radial in the following sense: fix an initial time variable $\bar{t}_0$, then the function $u_{\bar{t}_0}(\bar{x})=u(\bar{t}_0,\bar{x})$ is radial in the space variable, for $u$ it can be factorized through the isoparametric function $\varphi(\bar{t},\bar{x})=-\vert t\vert^2 + \vert x\vert^2$, where $\vert\cdot\vert$ denotes the norm in the usual Euclidean space. The same is true fixing now the space variable instead of the time one. 
	\item The solutions having parabolic cylinders as level subsets are new and they are not radial when fixing the time or the space variables, as one can see in the example constructed in Appendix \ref{App:Minkowski}. The global solutions are not stable for $\vert u\vert\nrightarrow 0$ asymptotically and cannot be a solution to the critical energy wave equation as defined in \cite{Ke}. However, as this example follows from the existence of parabolic cylinders as isoparametric hypersurfaces, it is interesting to ask whether other non trivial examples of hypersurfaces, distinct from the ones given here (for instance, the $B$-scroll hypersurfaces), could be used to obtain new kinds of well defined solutions to the critical energy wave equation. 
\end{enumerate}

The last remark lead us to the following open questions, which may be interesting for both the geometric and the analytical points of view:

\begin{question}
Given an isoparametric hypersurface $S$ in the semi-Riemannian Euclidean space, does there exists a globally defined isoparametric function $\varphi:\R_s^m\rightarrow\R$ such that $S=\varphi^{-1}(c)$ for some regular value $c$? If not, under which conditions is this possible? 
\end{question}

Next we tackle the Yamabe problem on the pseudosphere. The following result is a generalization of the main theorems in \cite{FePaPe,FePe}. To make a precise statement, we point out that principal curvatures of the isoparametric hypersurfaces in Example \ref{Ex:IsoparametricDeSitter} have, at most, two multiplicities $m_1$ and $m_2$ (see Appendix \ref{App:DeSitter}). Define $n_i:= (m-1)-m_i$, $i=1,2$ and let $\kappa:=\min\{n_1,n_2\}$.

\begin{theorem}\label{Th:NodalDeSitterIsop} Let $k\geq 2$ be a positive integer and  $S$ be an isoparametric hypersurface given by the functions in Example \ref{Ex:IsoparametricDeSitter}. If  
$
p<\frac{(m-\kappa)+2}{(m-\kappa)-2},
$
then there exist isoparametric hypersurfaces $S_+$, $S_0$ and $S_{-}$ (possibly $S_+=S_{-}$), and a blowing-up, sign changing solution $u_k$ to Equation \eqref{Eq:MainYamabeSphereNorm} with exactly $k$ nodal domains such that its nodal set has exactly $k-1$ connected components diffeomorphic to $S_0$, $u(z)\rightarrow\infty$ as $\text{dist}(z,S_{+})\rightarrow 0$ and $(-1)^{k-1}u(z)\rightarrow\infty$ as $\text{dist}(z,S_{-})\rightarrow 0$. Moreover, $S$ and every connected component of a nonempty regular level set are homothetic  to $S_0$, $S_{-}$ or $S_{+}$, and every connected component of the critical set is homothetic to one of these hypersurfaces or to one of the connected components of the focal submanifolds.
\end{theorem}

We point out that in the Riemannian setting ($s=0$), the numbers $n_1$ and $n_2$ are the dimensions of the focal varieties. However, this is no longer true for positive signatures. In Appendix \ref{App:DeSitter} we will describe the properties of these isoparametric hypersurfaces and their focal varieties in order to give concrete examples of this phenomenon and to obtain a better image of the solutions obtained in the theorem.

In case of the critical exponent $p=p_m$, the only way to attain the subcriticality in order to produce a solution to the Yamabe problem on the pseudosphere \eqref{Eq:MainYamabeSphere} is when $\kappa>0$. As this is not provided in the linear examples, we have the following immediate consequence of Theorem \ref{Th:NodalDeSitterIsop}.

\begin{corollary}
The Yamabe equation \eqref{Eq:MainYamabeSphere} admit sign-changing blowing-up solutions with a prescribed number of nodal domains, and having as level sets the isoparametric hypersurfaces described by the quadratic and Clifford isoparametric functions in Example \ref{Ex:IsoparametricDeSitter}.
\end{corollary}

\smallskip

The proof of main Theorems \ref{Th:MinkowskiIsop} and \ref{Th:NodalDeSitterIsop} relies on the reduction method given by Proposition \ref{Prop:ReductionODE}. We will show that both problems \eqref{Eq:MainYamabeTypeMinkowski} and \eqref{Eq:MainYamabeSphereNorm} can be reduced to the study the qualitative properties of the solutions to a generalized Emden-Fowler equation of the form
\begin{equation}\label{Eq:GeneralEmdenFowler}
w''+q(t)w'=f(w),\qquad\text{ in }\ I
\end{equation}
with initial condition $w'(0)=0$, where $q$ and $f$ are suitable continuous functions whose properties will be defined later on in Section \ref{Sec:GeneralEmdenFowler}, and $I$ is either $[0,\infty)$ or $[0,\pi]$. The reduction will be performed, separately, for equations \eqref{Eq:MainYamabeTypeMinkowski} and \eqref{Eq:MainYamabeSphereNorm} in Section \ref{Sec:Redcution Method}, while in Section \ref{Sec:GeneralEmdenFowler} we will develop the theory of global existence and blow-up of solutions to equation \eqref{Eq:GeneralEmdenFowler} and study their qualitative behavior. In particular, in \ref{Sec:Proof 1} we present the proof of  Theorems \ref{Th:MinkowskiIsop} and \ref{Th:GeometrySolutions} as a consequence of a more  general result for the reduced equation, stated in Theorem \ref{Th:SolutionsMinkowskiODE} below. Finally, in Section \ref{Sec:Proof 2} we state and prove a general result for the reduced equation obtained from the isoparametric reduction on the pseudosphere, Theorem \ref{Th:SolutionsDeSitterEntireODE} below, from which the proof of Theorem \ref{Th:NodalDeSitterIsop} follows immediately. We add two Appendixes about the geometry of the level sets of some isoparametric functions on the semi-Euclidean space and on the pseudosphere, in order to get a better understanding of the solutions to the main problems, and we also present an Appendix about the local existence and uniqueness to some singular differential equations of Emden-Fowler type.


\section{The reduction method}\label{Sec:Redcution Method}

In this section we will reduce equations \eqref{Eq:MainYamabeTypeMinkowski} and \eqref{Eq:MainYamabeSphereNorm} into \eqref{Eq:GeneralNonlinearODE} by means of the isoparametric functions in Examples \ref{Ex:IsoparametricMinkowsky} and \ref{Ex:IsoparametricDeSitter} respectively, and then we will fit \eqref{Eq:GeneralNonlinearODE} into a generalized Emden-Fowler equation having the form \eqref{Eq:GeneralEmdenFowler} in each case.

\subsection{The case of the Minkowski space form.}\label{Section:ReductionMinkowski}

Let $\varphi:\R^m_s\rightarrow\R$ be the isoparametric function  given in Example \ref{Ex:IsoparametricMinkowsky}. Hence, $\varphi$ satisfies \eqref{Eq:IsoparamFunctionMinkowski} and by Proposition \ref{Prop:ReductionODE}, the Yamabe type problem \eqref{Eq:MainYamabeTypeMinkowski} reduces to
\begin{equation}\label{Eq:MainReducedYamabeMinkowski}
(\gamma t+\delta)v'' + \beta v' =-\mu \vert v\vert^{p-1}v \quad\text{ in }\text{Im}\varphi,
\end{equation}
where 
\[
\gamma:= 4\alpha, \delta:= 4\langle a, a \rangle \text{ and }\beta:= 2\text{tr}A
\]
Notice that $\gamma,\delta,\beta\in\R$ could take any value, for the semi-Riemannian metric is not positive definite, 

As it is stated in Appendix \ref{App:Minkowski}, this kind of isoparametric functions may have different images, yielding very different ODE's. For this reason, we will only consider some of the main examples here with their corresponding reduced equation.
\subsubsection{Case $A=0$.} The isoparametric function is just $\varphi(z) = 2\langle a, z \rangle$ and is not constant if and only if $a\neq 0$; the isoparametric hypersurfaces are parallel hyperplanes. In this case  Im$\varphi=\R$ and as $\beta=0=\gamma$, so that the equation \eqref{Eq:MainReducedYamabeMinkowski} with $\delta\neq0$ is simply
\begin{equation}\label{Eq:MainReduction A=0}
 v'' =-\frac{\mu}{\delta} \vert v\vert^{p-1}v \quad\text{ in }\R.
\end{equation}

When $a=(-1,0,\ldots,0)$, the function $\varphi$ is just the projection onto the first factor. This is the kind of simplification used in \cite{Gi,KoLi} to obtain solutions to the Yamabe problem not depending on the space variable.

In order to analyze this equation, we are led to study it separately in $(-\infty,0]$ and in $[0,\infty)$. Observe that if $t=-s$ and we take $y(s)=v(t)$ for $s\in[0,\infty)$, then solving \eqref{Eq:MainReduction A=0} in $(-\infty,0]$ is equivalent to solving
\[
y''=- \frac{\mu}{\delta} \vert y\vert^{p-1}y \quad\text{ in }[0,\infty).
\]

Hence, we are led to just study the Emden-Fowler equation
\[
w''= \Lambda \vert w\vert^{p-1}w \quad\text{ in }[0,\infty).
\]
with $\Lambda\in\R\smallsetminus\{0\}$. We postpone this until the next section.


\subsubsection{Case $A\neq 0$ and $\gamma=0$.} Now the only eigenvalue is $\alpha=0$, $a$ is a eigenvector satisfying $\langle a,a \rangle \ne 0$ and $A$ es nilpotent, with $A^2=0$. We will suppose also that $\beta=$tr$A=0$ and that Im$\varphi=\R$, for the known examples satisfy this property. An explicit example, generalizing the one given in \cite{Ha} to arbitrary signature, is given in Appendix \ref{App:Minkowski}. In this case the equation reduces to

\begin{equation}\label{Eq:Main A neq 0 gamma=0}
v''  =-\frac{\mu}{\delta} \vert v\vert^{p-1}v \quad\text{ in }\R,
\end{equation}

which has the same form as \eqref{Eq:MainReduction A=0} and will also be study in the next section.

\subsubsection{Case $A\neq0$ and $\gamma\neq 0$.} The conditions on $A$ imply that the minimal polynomial of $A$ is $t, t-\alpha$ or $(t-\alpha)t$, where we obtain that $A$ is diagonalizable with proper values $\alpha\neq 0$ and $0$, and proper vector $a$ associated to $\alpha$. In what follows, we will take $a=0$, for the linear term in the expression of $\varphi$ just gives a translation of its isoparametric hypersurfaces. Hence $\delta=0$ and the equation \eqref{Eq:MainReducedYamabeMinkowski} can be reduced to
\begin{equation}\label{Eq:MainRedution A neq 0 gamma neg 0}
 tv'' + \frac{(k+n)}{2} v' =-\frac{\mu}{4\alpha} \vert v\vert^{p-1}v \quad\text{ in }\text{Im}\varphi,
\end{equation}
where $0\leq k\leq s$ and $0\leq n\leq m-s$ are integers satisfying $k+n>0$, and Im$\,\varphi$ could be either $(-\infty,0]$, $[0,\infty)$ or $\R$, according to the values of $k+n$ and $\alpha$ (see Appendix \ref{App:Minkowski}).

Observe that the change of variables $t=-s$, $y(s)=v(t)$, transforms the equation defined in $(-\infty,0]$ to a one defined in $[0,\infty)$ having the form
\[
sy'' + \frac{(k+n)}{2} y' =\frac{\mu}{4\alpha} \vert y\vert^{p-1}y \quad\text{ in }[0,\infty),
\]
so we are led to study the more general equation 
\begin{equation}\label{Eq:Main Aneq0 gamma neq 0 general}
tv''+\eta v' = \vartheta \vert v\vert^{p-1}v \qquad \text{ in }[0,\infty)
\end{equation}
where $\eta>\frac{1}{2}$ and $\vartheta\in\R\smallsetminus\{0\}$.

Yet we can consider another change of variables to get rid of the term in front of the higher derivative. For $t\geq 0$, consider $t=r^2$ and $w(r):=v(t)$, hence $v$ solves \eqref{Eq:Main Aneq0 gamma neq 0 general} if and only if $w$ solves
\begin{equation}\label{Eq:Main Aneq0 gamma neq 0 Classical}
w''+\frac{\theta}{r} w' = -\Lambda \vert w\vert^{p-1}w \qquad \text{ in }[0,\infty)
\end{equation}
where now $\theta=2\eta-1>0$ and $\Lambda=-4\vartheta\in\R\smallsetminus\{0\}$. This equation fits into a more general one that will be treated in Section \ref{Sec:GeneralEmdenFowler}.

\subsection{The case of the pseudosphere.}\label{Section:ReductionDeSitter}

In what follows, $\grad^S$ and $\square^S$ will denote the gradient and the Laplacian of a function restricted to $\S^{m}_s$ with its corresponding semi-Riemannian norm, while $\grad^L$ and $\square^L$ denote the corresponding gradient and Laplacian in $\R_s^{m+1}$. We will see first that all the isoparametric functions in Example \ref{Ex:IsoparametricDeSitter} fit into a general setting.

We say that an homogeneous polynomial $\Phi:\R_s^m\rightarrow\R$ of degree $\ell$ is a \emph{semi-Riemannian Cartan-M\"{u}nzner polynomial} if it satisfies the the differential equations
\begin{align}
\langle \grad^L\Phi(z),\grad^L\Phi(z) \rangle = \ell^2\langle z,z \rangle^{\ell-1}\\
\square^L\Phi(z)=\ell^2\beta\langle z,z \rangle^{\frac{\ell-2}{2}}
\end{align}
where $\beta\geq 0$ is a constant. 

The functions in Example \ref{Ex:IsoparametricDeSitter} are given by semi-Riemannian Cartan-M\"{u}nzner polynomials. Indeed, in the linear example $\ell=1$ and $\beta=0$. If $$A=\left(\begin{tabular}{cc}
$-I_{k_1}$ & $0$\\
$0$ &  $I_{k_2}$
\end{tabular}\right)=\text{diag}(\underbrace{-1,\ldots,-1}_{k_1},\underbrace{1,\ldots,1}_{k_2}),$$
with $k_1+k_2=m+1$, $k_1>s$, the quadratic examples are also given by this kind of polynomials and satisfy that $\ell=2$ and $\beta=\frac{k_2-k_1}{2}$. For the Clifford we have that $\ell=4$  and $\beta=\frac{k+1-2n}{2}$ (see Appendix \ref{App:DeSitter}).

In the Riemannian setting, the restriction of a Cartan-M\"{u}nzner polynomial to the round sphere $\S^m$ gives an isoparametric function. We next show that also any semi-Riemannian Cartan-M\"{u}nzner polynomial, when restricted to the De Sitter space $\S_s^m$, gives a semi-Riemannian isoparametric function. First we need a semi-Riemannian version of Euler's Theorem.  In what follows, $\grad^{E}$ will denote the gradient of a function with respect to the Euclidean metric  $\langle\cdot,\cdot\rangle_E$, while the semi-Riemannian metric in $\R_s^{m+1}$ will be denote simply by $\langle \cdot,\cdot\rangle$.

\begin{lemma}[Semi-Riemannian Euler's Theorem] If $\Phi:\R^{m+1}\rightarrow\R$ is a homogeneous polynomial of degree $\ell$, then
	\[
	\langle \grad^{L}\Phi,z \rangle= \langle \grad^{E}\Phi,z \rangle_E = \ell \Phi(z).
	\]
\end{lemma}

\begin{proof}
As $\Phi$ is a homogeneous polynomial of degree $\ell$, then
\begin{equation*}
\Phi(tz)=t^\ell\Phi(z)
\end{equation*}
Now, on the one hand, taking the derivative with respect to $t$ in this identity we have that
	\[
	\sum_{i=1}^{m+1}\frac{\partial \Phi}{\partial z_i}(tz)z_i=\ell t^{\ell-1}\Phi(z).
	\]
	taking $t=1$ and recalling that $\grad^{E}\Phi=\sum_{i=1}^{m+1}\frac{\partial \Phi}{\partial z_i}$, we obtain
	\[
	\langle \grad^{E}\Phi, z\rangle_E =\sum_{i=1}^{m+1}\frac{\partial \Phi}{\partial z_i}(z)z_i=\ell \Phi(z).
	\]
	On the other hand, if $\{e_1,\ldots,e_{m+1}\}$ is the canonical orthonormal frame of $\R^{m}_s$, then
	\[
	\grad^L\Phi=\sum_{i=1}^{m+1}\varepsilon_i\frac{\partial \Phi}{\partial z_j}e_j
	\]
	where, as usual $\varepsilon_i:=\langle e_i, e_i \rangle=-1$ if $1\leq i\leq s$ and  $\varepsilon_i=1$ if $s+1\leq i\leq m+1$.
	Hence, for any given $z=\sum_{i=1}^{m+1}\varepsilon_i z_i$ we have that
	\[
	\langle \grad^L\Phi,z \rangle=\sum_{i,j=1}^{m+1}\varepsilon_i\frac{\partial \Phi}{\partial z_i}z_j\langle e_i, e_j\rangle=\sum_{i=1}^{m+1}(\varepsilon_i)^2\frac{\partial \Phi}{\partial z_i}z_i
	=\langle \grad^{E} \Phi, z\rangle_E= \ell \Phi(z)
	\]
as we wanted.
\end{proof}

\begin{lemma}
	If $\Psi:\R_s^m\rightarrow\R$ is a semi-Riemannian Cartan-M\"{u}nzner polynomial, then the restriction $\varphi:=\Phi\upharpoonright:\S_s^m\rightarrow\R$ is isoparametric and
	\begin{equation}
	\langle \grad^S\varphi,\grad^S\varphi \rangle = \ell^2(1-\varphi^2)\quad\text{and}\quad\square\varphi=\ell^2\beta - \ell(m+\ell-1)\varphi 
	\end{equation}
\end{lemma}

\begin{proof}
Using the semi-Riemannian Euler's Theorem, the proof is, line by line, the same as the proof of Theorem 3.30 in \cite{CeRy} and we omit it.
\end{proof}

This result allow us to give a concrete description of the functions $a$ and $b$ given in \eqref{Eq:Isoparametric} in terms of $\ell$, $\beta$ and $n$. Indeed, if $\varphi$ is the restriction to the pseudosphere of a semi-Riemannian Cartan-M\"{u}nzner polynomial, then
\begin{equation}\label{Eq:CartanMunznerExamples}
a(t)=\ell^2\beta-\ell(m+\ell-1)t \quad\text{and}\quad b(t)=\ell^2(1-t^2)
\end{equation}

In what follows, we will suppose that 
\begin{equation}\label{Eq:TechnicalSuppositionBeta}
\frac{m-1}{\ell}\pm\beta>0
\end{equation}
This condition is satisfied by the functions given in Example \ref{Ex:IsoparametricDeSitter}. 

In the Riemannian case, the associated isoparametric hypersurfaces have $\ell\in\{1,2,3,4,6\}$ different principal curvatures and at most two different multiplicities, say $m_1$ and $m_2$ \cite{Mu}. In this scenario, the constant is $\beta=\frac{m_1-m_2}{2}$ and condition \eqref{Eq:TechnicalSuppositionBeta} is always satisfied, for $\frac{m-1}{\ell}\pm\beta=m_i$, $i=1,2$. The same happens for the hypersurfaces defined by the isoparametric function in Example \ref{Ex:IsoparametricDeSitter}, see Appendix \ref{App:DeSitter}. Also, the Cartan-M\"{u}nzner polynomials fully characterize the isoparametric hypersurfaces on the sphere, in the sense that if $S$ is an isoparametric hypersurface, there exists an isoparametric function, say $\varphi$, having it as regular level set and there exists a Cartan-M\"{u}nzner polynomial giving the same isoparametric foliation as $\varphi$, when restricted to $\S^m$. The semi-Riemannian isoparametric hypersurfaces are not that well behaved. As it was shown by Hahn in \cite{Ha}, there are other examples of isoparametric hypersurfaces which are not generated by semi-Riemannian Cartan-M\"{u}nzner polynomials, for instance, the so called \emph{totally isotropic hypersurfaces.} However, this function is also given by the restriction of a homogeneous polynomial. Moreover, Li showed the existence of a great variety of isoparametric hypersurfaces for the Lorentzian De Sitter space \cite{Li}, and as in the case of the semi-Riemannian Euclidean space, the existence of a global isoparametric function defining them is not clear. This leads to the following open questions.

\begin{question}
Given an isoparametric hypersurface $S$ on $\S^m_s$, does there exists a global isoparametric function having it as a level set? Are they all given as a restriction of a homogeneous polynomial?
\end{question}

Also the examples given by semi-Riemannian Cartan-M\"{u}nzner polynomials are far for been fully understood.  We have the following.

\begin{question}
If $\varphi$ is an isoparametric function given as te restriction of a Cartan-M\"{u}nzner polynomial and $S$ is an isoparametric hypersurface determined by $\varphi$, what values of $\ell$ are possible? Are there more than two multiplicities? 
\end{question}

\bigskip

We next reduce the PDE \eqref{Eq:MainYamabeSphereNorm} into the ODE \eqref{Eq:GeneralNonlinearODE} using restrictions of semi-Riemannian Cartan-M\"{u}nzner polynomials.

Let $\varphi:\S_s^m\rightarrow\R$ be given by the restriction of a semi-Riemannian Cartan-M\"{u}nzner polynomial, so that the functions $a$ and $b$ satisfying \eqref{Eq:Isoparametric} are given by \eqref{Eq:CartanMunznerExamples}. By Proposition \ref{Prop:ReductionODE} the Yamabe type equation \eqref{Eq:MainYamabeSphereNorm} takes the form
\begin{equation}\label{Eq:CartanMunznerEntireODE}
(1-t^2)v''+\left[-\frac{m+\ell-1}{\ell} t+\beta\right]v'+\frac{\lambda}{\ell^2}[\vert v\vert^{p-1}v-v]=0, \quad \text{ on }\text{Im}\;\varphi.
\end{equation}

The image of $\varphi$ could be $\R$, $[-1,\infty)$, $[,\infty)$, $(-\infty,-1)$ and $(-\infty,1]$, as shown in Appendix \ref{App:DeSitter}, being $\pm 1$ the only singularities of the equation, which are also the only critical values of the function $\varphi$. In what follows, we will focus on equation \eqref{Eq:CartanMunznerEntireODE} when Im$\,\varphi=\R$, being the other possibilities analogous but simpler. It is evident that the equations \eqref{Eq:CartanMunznerEntireODE} admits always the constant solutions $v\equiv0$ and $v\equiv\pm 1$ and that if $v$ is a solution, then also $-v$ is. As the only singular points are $t=\pm 1$, we will study the equation \eqref{Eq:CartanMunznerEntireODE} separately in $(-\infty,-1]$, $[-1,1]$ and $[1,\infty)$ 

Equation \eqref{Eq:CartanMunznerEntireODE} defined in the interval $[-1,1]$ has been extensively studied recently, and arise naturally when seeking for solutions to Yamabe and Brezis-Li-Nirenberg type problems on the Riemannian round sphere, see, for instance \cite{BeJuPe,BiVe,BrLi,FePaPe,FePe,HePe,JuPe}. The existence of multiple positive solutions for $\Lambda$ big enough was studied in \cite{BeJuPe,BrLi,HePe} using bifurcation methods, whilst the existence of sign-changing solutions was studied in \cite{FePe} for the critical exponent and in \cite{FePaPe,JuPe} for more general exponents, using a double shooting method introduced in \cite{FePe}. Letting $\beta:=\frac{n_1-n_2}{2}$, where $n_1,n_2\in\{0,1,\ldots,m-2\}$ are such that $\frac{(m-1)(\ell-1)}{\ell}=\frac{n_1+n_2}{2}$, and considering the new variables $t=\cos r$ and $w(r)=v(t)$, one can easily check that solving \eqref{Eq:CartanMunznerEntireODE} with its natural boundary conditions 
\begin{equation}\label{Eq:NaturalBoundaryConditions}
\left[-\frac{m+\ell-1}{\ell} (\pm 1)+\beta\right]v'(\pm 1)=-\frac{\lambda}{\ell^2}[\vert v(\pm1)\vert^{p-1}v(\pm 1)-v(\pm 1)]
\end{equation}
is equivalent to solving the problem
\begin{equation}\label{Eq:Compact}
w''+\frac{h_0(r)}{\sin r}w'+\frac{\lambda}{\ell^2}[\vert w\vert^{p-1}w-w]=0,\quad \text{ on }[0,\pi],
\end{equation}
with natural boundary conditions $w'(0)=0=w'(\pi)$, where $h_0(r)=\frac{m-1}{\ell}\cos r-\beta.$ Observe that the initial and final conditions satisfy $w(0)=v(1)$ and $w(\pi)=v(-1)$. The following theorem gathers the results given in \cite{FePaPe,FePe,BeJuPe,JuPe}.

\begin{theorem}\label{Th: k zeroes}
	Fix $\ell\in\{1,2,3,4,6\}$,  $m\geq 3$ and $n_1,n_2\in\{0,1,\ldots,m-2\}$ such that $\frac{(m-1)(\ell-1)}{\ell}=\frac{n_1+n_2}{2}$. Let $\kappa:=\min\{n_1,n_2\}$ and suppose $1<p<\frac{m-\kappa+1}{m-\kappa-1}$. Then there exist a sequence $0=\lambda_{0}<\lambda_1<\cdots<\lambda_j<\cdots$, depending only on $\ell,n_1,n_2,m$ and $p$, and satisfying that $\lambda_j\rightarrow\infty$ as $j\rightarrow\infty$, with the following property: 
	If $\lambda\in(\lambda_{j},\lambda_{j+1}]$, then the boundary value problem \eqref{Eq:Compact} has at least $j$ non constant positive solutions. Moreover, for each $k\in\N$, there exist numbers $d_k$, depending only on $\ell,n_1,n_2,p$ and $\lambda$, with $1<d_1<\cdots<d_k<d_{k+1}<\cdots$ and $d_k\rightarrow\infty$ as $k\rightarrow\infty$, such that the problem \eqref{Eq:Compact}
	admits a solution $w_k$ having exactly $k$ zeroes in $(-1,1)$, initial condition $w_k(0)=d_k$ and final condition $w_k(\pi)>1$ for $k$ even and $w_k(\pi)<-1$ for $k$ odd, satisfying that $\vert w_k(\pi)\vert\rightarrow\infty$ as $k\rightarrow\infty$.
\end{theorem} 

The condition $\frac{(m-1)(\ell-1)}{\ell}=\frac{n_1+n_2}{2}$ is equivalent to $\frac{m-1}{\ell}=\frac{m_1+m_2}{2}$, where $m_1$ and $m_2$ are the multiplicities of an isoparametric hypersurface in the Riemannian sphere and, in this case, $n_1$ and $n_2$ are the dimensions of the focal submanifolds associated to the isoparametric family \cite{FePaPe}. 
Notice that the first assertion of the theorem says that if $\lambda<\lambda_1$, then the only positive solution is the constant one. When $\ell=1$, $1<p\leq\frac{m+1}{m-1}$, $\lambda_1$ is given explicitly by $\lambda_1=\frac{m}{p-1}$, and the case $p=\frac{m+1}{m-1}$ and $\lambda_1=\frac{m(m-2)}{4}$ correspond to the axially symmetric solutions of the Yamabe problem \cite{BiVe,BrLi}. It may be interesting to give explicit expression for both $\lambda_j$ and $d_k$ and explore the relations of these numbers with the eigenvalues of the Laplacian on the sphere and the $k$-th Yamabe invariants on the same manifold. A partial result for $\lambda_j$ can be found in \cite{BeJuPe}.

On the other hand, in \cite{BeJuPe,HePe}, it is not clear whether the initial and final conditions of the positive solutions satisfy $\vert w(0)\vert,\vert w(\pi) \vert <1$ or $\vert w(0)\vert,\vert w(\pi) \vert >1$.  Numerical evidence suggests that both cases are possible, contrasting with what happens when considering the sign-changing solutions, for which $\vert w_k(0)\vert,\vert w_k(\pi) \vert >1$ always. In addition, it is not clear  whether $\vert w(0)\vert < d_1$ or not for positive solutions.

\bigskip

Now we write, for simplicity, $\alpha:=-\frac{m+\ell-1}{\ell},$ $f(t):=\Lambda[\vert t\vert^{p-1}t-t]$ with $\Lambda:=\frac{\lambda}{\ell^2}$. We next consider the initial value problem 
\begin{equation}\label{Eq:CartanMunznerPositiveODE}
(1-t^2)v''+[\alpha t+\beta]v'+f(v)=0, \quad \text{ on }I_{\pm}
\end{equation}
with $v(\pm 1)=d>1$ and initial condition on the derivative given by \eqref{Eq:NaturalBoundaryConditions}, where the interval of definition is either $I_+:=[1,\infty)$ or $I_-:=(-\infty,-1]$. For $I_+$, by means of the new variable $t=\cosh r$ and $w(r)=v(t)$, we transform this equation into the Emden-Fowler type equation
\begin{equation}\label{Eq:PositiveODE}
w''+\frac{h_+(r)}{\sinh r}w'-f(w)=0, \quad \text{ on }[0,\infty).
\end{equation}
with initial conditions $w(0)=d$ and $w'(0)=0$, where $h_+(r):=[\frac{m-1}{\ell}\cosh r-\beta]$. On the other hand, for $I_-$, the new variable $t=\cosh(-r)$ and $w(r)=v(t)$ transform equation \eqref{Eq:CartanMunznerPositiveODE} with initial conditions on $t=-1$ into the same equation \eqref{Eq:PositiveODE} and the same initial conditions, but now with the function $h_-:(r):=[\frac{m-1}{\ell}\cosh r + \beta]$ instead of $h_+$. Notice that $h_{\pm}$ is positive and monotone increasing because of \eqref{Eq:TechnicalSuppositionBeta}. Equation \eqref{Eq:PositiveODE} fits into the general form of equation \eqref{Eq:GeneralEmdenFowler}, which is the main topic of the next section.


\section{Existence of global and blowing-up solutions to some generalized Emden-Fowler equations}\label{Sec:GeneralEmdenFowler}

We will focus our attention in the following initial value problem

\begin{equation}\label{Eq:GeneralIVP}
\left\{
\begin{tabular}{cc}
$w''+q(r)w'=\pm f(w)$ & on $[0,\infty)$\\
$w(0)=a,$ $w'(0)=0$ & $a\geq 0$
\end{tabular}
\right.\tag{$A\pm$},
\end{equation}
with $f:\R\rightarrow\R$ an odd and locally Lipschitz continuous function satisfying
\begin{align}
&\lim_{t\rightarrow\infty}\frac{f(t)}{t}=\infty, \tag{$f 1$}\label{Hyp:f1}\\
&\exists\  t_0\geq 0 \;  \text{ s.t. } \; f(t_0)=0  \text{ and } f \text{ is strictly increasing in } (t_0,\infty),\tag{$f 2$}\label{Hyp:f2}\\
&\text{ if }F(t):=\int_{0}^t f(s)ds, \text{ and }F(t)\geq 0 \text{ for every }t\geq t_1\geq t_0, \text{ then } \tag{$f 3$}\label{Hyp:f3}\\
&\quad\quad\quad\quad\quad\quad\quad\quad\quad\quad\int_{t_1+1}^{\infty}\frac{ds}{\sqrt{F(s)}}<\infty;\nonumber 
\end{align}
and $q:(0,\infty)\rightarrow(0,\infty)$ continuously differentiable with the additional properties
\begin{align}
&\lim_{r\rightarrow0}q(r)=\infty\quad \text{ and }\quad \lim_{r\rightarrow 0}\int_{r}^1 q(s)ds=\infty \tag{$q 1$}\label{Hyp:q1}\\
&\lim_{r\rightarrow\infty}q(r)\geq 0 \text{ exists. } \tag{$q 2$}\label{Hyp:q2}\\
&\lim_{r\rightarrow\infty}\frac{q(r)}{r^\alpha}>0 \text{ for some }\alpha\geq -1. \tag{$q 3$}\label{Hyp:q3}\\
&\lim_{r\rightarrow0}-\frac{q^2(r)}{q'(r)}>0\text{ exists. }\tag{$q 4$}\label{Hyp:q4}
\end{align}

Observe that condition \eqref{Hyp:f1} is just saying that $f$ is superlinear. Condition \eqref{Hyp:q1} says that $q\notin L^1([0,1])$ while \eqref{Hyp:q2} says that it is bounded away from zero. Condition \eqref{Hyp:q4} will guarantee the existence and uniqueness of the solution at the singularity $r=0$. Even if all these conditions appear to be very restrictive, there are interesting examples satisfying hypotheses \eqref{Hyp:f1}-\eqref{Hyp:f3} and \eqref{Hyp:q1}-\eqref{Hyp:q4}. For instance, the functions $f(t)=\Lambda\vert t\vert^{p-1}t$ and $q(r)=\frac{\theta}{r}$ with $p>1$ and $\theta,\Lambda>0$,  appearing in equation \eqref{Eq:Main Aneq0 gamma neq 0 Classical}, satisfy them. The same is true for the functions $f(t)=\Lambda[\vert t\vert^{p-1}t-t]$ and $q(r)=\frac{\alpha\cosh{r}\pm\beta}{\sinh(r)}$   with $\Lambda>0$ and $\alpha\pm\beta>0$, appearing in equation \eqref{Eq:PositiveODE}.  Observe that the more general nonlinearity $f(t)=\Lambda\vert t\vert^{p-1}t-\delta \vert t\vert^{s-1}t$ with $1\leq s<p$ and $\Lambda,\delta>0$, also satisfies the conditions \eqref{Hyp:f1}-\eqref{Hyp:f3}.

If $\rho:[0,\infty)\rightarrow[0,\infty)$ is an integrating factor for equation \eqref{Eq:GeneralIVP}, then 
\begin{equation}\label{Eq:IntegratingFactorProperty}
\frac{\rho'(r)}{\rho(r)}=q(r)\quad\text{for each }\ r>0.
\end{equation}
and we can rewrite the equation in divergence form as

\begin{equation}\label{Eq:GeneralDivergenceIVP}
\left\{\begin{tabular}{cc}
$(\rho(r)w')'=\pm \rho(r) f(w)$ & on $[0,\infty)$\\
$w(0)=a,$ $w'(0)=0$ & $a\geq 0$
\end{tabular}\right.
\tag{$B\pm$},
\end{equation}

Condition \eqref{Hyp:q1}-\eqref{Hyp:q3} imply the following properties of the integrating factor.

\begin{lemma}\label{Lemma:PropertiesIntegratingFactor}
The integrating factor is monotone increasing with $\rho(0)=0$, $\lim_{r\rightarrow\infty}\rho(r)=\infty$ and
\begin{equation}\label{Eq:LimitIntegralFactor}
\lim_{r\rightarrow\infty}	\frac{\int_{0}^r\rho(s)ds}{\rho(r)}>0
\end{equation}	
\end{lemma}

\begin{proof}
	An integrating factor for equation \eqref{Eq:GeneralIVP} is given by
	\[
	\rho(r):=e^{\int_{1}^{r}q(s)ds},\quad r\in(0,\infty)
	\]
	since it satisfies \eqref{Eq:IntegratingFactorProperty}
	
	As $q>0$, $\rho$ is strictly increasing and positive. By \eqref{Hyp:q1}, we have that
	\[
	\lim_{r\rightarrow0}\rho(r)=\lim_{r\rightarrow 0}e^{\int_{1}^{r}q(s)ds}=\lim_{r\rightarrow 0}e^{-\int_{r}^{1}q(s)ds}=0.
	\]
	Now we prove that $\rho$ is unbounded. By \eqref{Hyp:q3}, there exists $D>0$ and $r_1>1$ such that $\frac{q(r)}{r^\alpha}\geq D>0$ for every $r\geq r_1$. Hence
	
\[
\lim_{r\rightarrow\infty}\rho(r) = e^{\int_{1}^{r_1}q(s)ds}e^{\lim_{r\rightarrow\infty} \int_{r_1}^{r}q(s)ds} \geq e^{\int_{1}^{r_1}}e^{D\int_{r_1}^{\infty}s^\alpha ds}=\infty,
\]
because $D>0$ and $\alpha\geq -1$. Finally \eqref{Eq:LimitIntegralFactor} follows directly from the last limit, L'H\^{o}pital's rule, \eqref{Eq:IntegratingFactorProperty} and property \eqref{Hyp:q2}.
\end{proof}

We will consider the following additional hypotheses on $\rho$:

\begin{equation}
\text{there exists } N>0 \text{ s.t. }\frac{1}{\rho(t)}\int_0^t\rho(s)ds\leq N\text{ for every }t\in[0,1]   \tag{$\rho 1$}\label{Hyp:rho}
\end{equation}

This condition, together with \eqref{Hyp:q4}, allows us to prove the local existence and uniqueness to problem \eqref{Eq:GeneralIVP} with natural initial conditions $w(0)=a$ and $w'(0)=0$. To the reader convenience, we prove this fact in Appendix \ref{App:ExistenceUniquenessSingular}. Notice that the integrating factors $\rho(r)=r^\theta$ and $\rho(r)=C(\sinh(r/2))^{\alpha\pm \beta}(\cosh(r/2))^{\alpha\mp\beta}$ for equations 
\eqref{Eq:Main Aneq0 gamma neq 0 Classical} and \eqref{Eq:PositiveODE}, respectively, satisfy condition \eqref{Hyp:rho}, where $C>0$ is a constant.

\bigskip

In what follows, let $0<R\leq\infty$ be such that $[0,R)$ is the maximal interval of existence of a solution to \eqref{Eq:GeneralIVP}. What we will show next is that if we consider the minus sign in \eqref{Eq:GeneralIVP} or in \eqref{Eq:GeneralDivergenceIVP}, then all solutions are globally defined, i.e. $R=\infty$, independently of the initial condition $a\geq 0$, while if we choose the plus sign in the equation and consider $a>t_0$, then all solutions will blow-up in finite time, i.e., $R<\infty$. Obviously in both situations, the stationary solution $w(t)\equiv t_0$ satisfies the problem. 

\subsection{Existence of global and blowing-up solutions.}

We first prove that the solutions to the problem with the ``$-$" sign:
\begin{equation}\label{Eq:GeneralIVPNegative}
\left\{
\begin{tabular}{cc}
$w''+q(r)w'=- f(w)$ & on $[0,\infty)$\\
$w(0)=a,$ $w'(0)=0$ & $a\geq 0$
\end{tabular}
\right.\tag{$A-$},
\end{equation}
are globally defined.

\begin{proposition}\label{Prop:Global existence}
If $w$ is a solution to the problem \eqref{Eq:GeneralIVPNegative} $a\geq 0$, then $R=\infty$. 
\end{proposition}

\begin{proof}
If $w$ is a solution, define the energy function
\[
E(r):=\frac{(w'(r))^2}{2}+ F(w(r)),\quad r\in[0,R)
\]
Notice that $E'(r)=-q(r)\vert w'(r)\vert^2\leq 0,$ because $q>0$. Therefore $E$ is decreasing and $E(r)\leq E(0)$ for every $r\in(0,R)$. As $F$ is even, continuous and $F(t)\rightarrow\infty$ as $t\rightarrow\infty$, the last inequality for the energy implies that $w$ and $w'$ must be uniformly bounded in $(0,R)$. It is now routinary to check that $R=\infty$.
\end{proof}

We now focus on equation \eqref{Eq:GeneralIVP} with the ``plus'' sign and initial condition $a>t_0$, i.e.,
\begin{equation}\label{Eq:GeneralIVPPositive}
\left\{
\begin{tabular}{cc}
$w''+q(r)w'=+ f(w)$ & on $[0,\infty)$\\
$w(0)=a,$ $w'(0)=0$ & $a> t_0$
\end{tabular}
\right.\tag{$A+$},
\end{equation}
From now on, $w$ will denote a solution to this problem. We will prove that for these initial conditions the solution must blow-up in finite time, i.e., we will see that $R<\infty$ and $w(r)\rightarrow \infty$ as $r\rightarrow R$. To do so, we need some preliminary lemmas.

\begin{lemma}\label{Lemma:ww' Positive}
The solution $w$ satisfies that $ww'>0$ in $(0,R)$. In particular $w',w>0$ in $(0,R)$
\end{lemma}

\begin{proof}
By hypothesis \eqref{Hyp:q4}, if $\Gamma:=\lim_{r\rightarrow0}-\frac{q^2(r)}{q'(r)}>0$, we have that
\begin{equation}\label{Eq:InitialConditionSecondDerivative}
w''(0)=\frac{1}{1+\Gamma}f(a)>0,
\end{equation} 
Thus, continuity of $w''$ gives the existence of some $r_0>0$ such that $w''(r)>0$ for every $r\in[0,r_0)$ and $w'$ is monotone increasing in the same interval. As $w'(0)=0$, this implies that $w'>0$ in $(0,r_0)$, yielding at the same time that $w(r)>w(0)=a>t_0$. Hence $ww'>0$ in $(0,r_0)$. Notice that if it happens that $w'(r_1)=0$ for some $r_1>r_0$, then as $w'\geq 0$  in $[r_0,r_1]$, we would have that $w(r_1)\geq w(r_0)>t_0$ and, therefore, $w''(r_1)=f(w(r_1))>0$ again; so, the same argument yields that $ww'>0$ in an interval $(r_1,r_2)$ for some $r_2>r_1$. Therefore $w'$ can not change sign, $w'w\geq 0$ in $(0,R)$, $w$ is not decreasing and $w(r)\geq w(0)=a>t_0$. To conclude, integrate in (B+) from $0$ to $r$ and observe that
\[
w'(r)=\frac{1}{\rho(r)}\int_{0}^r\rho(s) f(w(s)) ds >0
\]
for every $r>0$, because $\rho>0$  and $f(w(r))>0$  by \eqref{Hyp:f2}.
\end{proof}

\begin{lemma}\label{Lemma:Limit w}
If $R=\infty$, then
\[
\lim_{r\rightarrow\infty}w(r)=\infty.
\]
\end{lemma}

\begin{proof}
From the previous lemma, $w$ is monotone increasing in $[0,\infty)$ and, as $f$ is also monotone increasing in $(t_0,\infty)$, we have that $f(w(r))\geq f(w(0))=f(a)$. From this and equation (B+) we obtain
\[
(\rho w')'=\rho f(w)\geq \rho f(a)\quad\text{ for every }r>0. 
\]
By \eqref{Eq:IntegratingFactorProperty}, there exist $C>0$ and $R_0>0$ such that $\int_0^r \rho(s)ds\geq C\rho(r)$, for every $r\geq R_0$. So, integrating from $0$ to $r$ we get that
\begin{equation*}
w'(r)\geq f(a)\frac{\int_0^r\rho(s)ds}{\rho}\geq f(a)C,
\end{equation*}
 for every $r>R_0$.
 
Integrating the last inequality from $R_0$ to $r$ we get that
\[
w(r)\geq w(r)-w(R_0)\geq f(a)C(r-R_0)>0,
\]
since $w(R_0),C,f(a)>0$ and since $w$ is monotone increasing. Taking limits on both sides as $r\rightarrow\infty$ we conclude.

\end{proof}

We now prove the blow-up of the solutions to \eqref{Eq:GeneralIVPPositive}. Our proof is a modification of the classic argument by Osserman \cite{Os}.

\begin{proposition}\label{Prop:Blow-Up}
If $w$ is a solution to \eqref{Eq:GeneralIVPPositive} with $w(0)=a>t_0$, then $R<\infty$ and $\lim_{r\rightarrow R}w(r)=\infty$.
\end{proposition}

\begin{proof}
Suppose, in order to get a contradiction, that $R=\infty$. By \eqref{Hyp:q2}, there exists $R_1>0$ and $M>0$ such that 
\begin{equation}\label{Hyp:q2prime}
q(r)\leq M \quad \text{for every }\ r\geq R_1.
\end{equation}
Given $a>t_0$ and
\begin{equation}\label{Eq:DefinitionN}
N:=8M^2>0,
\end{equation}
by \eqref{Hyp:f1}, there exists $t_2>a>t_0$ such that
\begin{equation}\label{Hyp:f1prime}
\frac{1}{N}f(t)\geq t-a\geq 0 \quad \text{for every }\ t\geq t_2.
\end{equation}

As $w'>0$ in $(0,\infty)$ by Lemma \ref{Lemma:ww' Positive}, $w$ is increasing and $w(r)>w(0)=a>t_0$ for every $r\in(0,\infty)$, implying that the function $f(w)$ is also monotone increasing in $(0,\infty)$. Since $q\geq 0$ and $w'>0$, from equation (B+) we have that
\[
w''\leq w'' + qw'=f(w) \quad \text{ in }\ [0,\infty).
\]
Multiplying by $w'>0$, we get that
\[
\left(\frac{(w')^2}{2}\right)'=w''w'\leq f(w)w' \quad \text{ in }\ [0,\infty).
\]
Therefore, integrating from $0$ to $r$, using that $w'(0)=0$ and noticing that $f(a)\leq f(t)\leq f(w(r)$ for every $t\in[a,w(r)]$ we deduce that
\begin{equation}\label{Eq:FundamentalInequalityBlowUp}
\begin{split}
(w'(r))^2&\leq 2\int_{0}^r f(w(s))w'(s)ds=2\int_{a}^{w(r)} f(t)dt\\
&\leq 2f(w(r))\int_{a}^{w(r)}ds = 2f(w(r))[w(r)-a].
\end{split}
\end{equation}
However, Lemma \ref{Lemma:Limit w} provide us the existence of $R_2>R_1$ such that $w(r)\geq \max\{t_2,t_1+1\}$ for every $r\geq R_2$, yielding by \eqref{Hyp:f1prime} that $\frac{1}{N}f(w(r))\geq w(r)-a$. On the one hand, As $R_2>R_1$, substituting this into \eqref{Eq:FundamentalInequalityBlowUp} we obtain that
\[
(w'(r))^2\leq 2f(w)[w-a]\leq\frac{2}{N}(f(w))^2.
\]
Taking square roots at both sides of the inequality and multiplying by $q>0$ we obtain
\begin{equation}\label{Eq:FundamentalInequalityBlowUp2}
q(r)w'(r)\leq \sqrt{\frac{2}{N}}\; q(r)f(w(r)),\quad\text{ for every }\ r\geq R_2.
\end{equation}
On the other hand, as $R_2>R_1$, then \eqref{Hyp:q2prime} holds true for every $r\geq R_2$ and, consequently from \eqref{Eq:FundamentalInequalityBlowUp2} we obtain
\[
q(r)w'\leq \sqrt{\frac{2}{N}}q(r)f(w(r))\leq M\sqrt{\frac{2}{N}}f(w(r))=\frac{1}{2}f(r)\quad\text{ for every }r\geq R_2
\]
by the definition of $N$ in \eqref{Eq:DefinitionN}. Equation \eqref{Eq:GeneralIVPPositive} and the last inequality yield that
\[
f(w)=w''+q(r)w'\leq w''+\frac{1}{2}f(w)\ \text{ in }[R_2,\infty)
\]
or, equivalently
\begin{equation}\label{Eq:FundamentalInequalityBlowUp3}
	0<\frac{1}{2}f(w(r))\leq w''(r)\quad \text{ for every }\ r\geq R_2
\end{equation}
As $w'>0$ in $[R_2,\infty)$, multiplying both sides of the inequality by $w'$, we obtain
\[
\frac{1}{2}f(w)w'\leq w'' w'=\left(\frac{(w')^2}{2}\right)'\quad\text{ in }[R_2,\infty).
\]
Integrating from $R_2$ to $r$ we get
\begin{equation}\label{Eq:FundamentalInequalityBlowUp4}
\int_{R_2}^rf(w(s))w'(s)ds\leq (w'(r))^2,\quad\text{for every }r\geq R_2
\end{equation}
Now observe that if $r\geq R_2$, then $f(w(s))w'(s)>0$ for every $s\in[R_2,r]$, hence
\begin{equation*}
0<\int_{R_2}^rf(w(s))w'(s)=\int_{0}^{w(r)}f(s) ds - \int_{0}^{w(R_2)}f(s)ds=F(w(r))-C_1(R_2),
\end{equation*}
where $C_1(R_2):=\int_{0}^{w(R_2)}f(s)$ is constant (not necessarily positive). Hence, \eqref{Eq:FundamentalInequalityBlowUp4} can be written as
\[
1\leq\frac{w'(r)}{\sqrt{F(w(r))-C_1(R_2)}}\quad\text{ for every }r\geq R_2,
\]

Integrating from $R_2$ to $r$  we have that
\[
r-R_2\leq\int_{w(R_2)}^{w(r)}\frac{ds}{\sqrt{F(s)-C_1(R_2)}}\leq\int_{w(R_2)}^\infty\frac{ds}{\sqrt{F(s)-C_1(R_2)}}
\]
As $r>R_2$ is arbitrary, letting $r\rightarrow\infty$ we get that $\int_{w(R_2)}^\infty\frac{ds}{\sqrt{F(s)-C_1(R_2)}}=\infty$. But this is a contradiction, since $\int_{w(R_2)}^\infty\frac{ds}{\sqrt{F(s)-C_1(R_2)}}<\infty$ because  $w(R_2)\geq t_1+1$ and  \eqref{Hyp:f3}. Hence $R<\infty$.

To finish, we prove that $w(r)\rightarrow\infty$ as $r\rightarrow R$. As $w$ is monotone increasing in $(0,R)$, the limit exists. Considering the energy $\hat{E}(r):=\frac{\vert w'(r)\vert^2}{2}-F(w(r))$, which is non increasing, for $\hat{E}'(r)=-q(r)\vert w'(r)\vert^2\leq 0$. Thus, if $\lim_{r\rightarrow R}w(r)<\infty$, both $\vert w\vert$ and $\vert w'\vert$ must be bounded in $[0,R)$ and a standard argument yields that the solution must exist in $[0,R+\varepsilon)$ for some $\varepsilon>0$, contradicting the maximality of $R$, and the limit follows. 
\end{proof}

We can apply this result to nonlinearities of the form $f(t)=\vert t\vert^{p-1}t-\vert t\vert^{s-1}t$, for $1\leq s<p$. We have the following

\begin{corollary}
Let $1\leq s<p$ and $\alpha,\beta,\theta\in\R$ such that $\alpha\pm\beta>0$ and $\gamma>0$. Then all the solutions to the problem
\[
\left\{\begin{tabular}{cc} 
$w''+q(r)w'= \vert w\vert^{p-1}w-\vert w\vert^{s-1}w$ & in $[0,\infty)$\\
$w(0)>1$,\quad $w'(0)=0$ & 
\end{tabular}\right.
\]	
with $q(r)=\frac{\alpha\cosh r \pm \beta }{\sinh r}$ or $q(r)=\frac{\theta}{r}$, are monotone increasing and blow up in finite time.
\end{corollary}

\begin{proof}
The functions $q$ and $f(t)=\vert t\vert^{p-1}t-\vert t\vert^{q-1}t$ satisfy the properties \eqref{Hyp:q1} to \eqref{Hyp:q4}, \eqref{Hyp:f1} to \eqref{Hyp:f3} and \eqref{Hyp:rho}.
\end{proof}


\subsection{Classical Emden-Fowler Equations and proof of Theorems \ref{Th:EmdenFowlerMinkowski} and \ref{Th:GeometrySolutions}}\label{Sec:Proof 1}

From the previous subsection, we described the qualitative behavior of the solutions to the problem \eqref{Eq:GeneralIVPPositive}, while we only showed the global existence of the solutions to the problem \eqref{Eq:GeneralIVPNegative}. In this section, we will focus on describing the qualitative behavior of the solutions to the last equation for the function $q(r)=\frac{\theta}{r}$ and the nonlinearity $f(t)=\Lambda\vert t\vert^{p-1}t$. Therefore, we are led to consider the equation
\begin{equation}\label{Eq:GeneralEmdenFowlerMinkowski}
\left\{\begin{tabular}{cc}
$w''+\frac{\theta}{r} w' = -\Lambda \vert w\vert^{p-1}w$   & in $[0,\infty)$\\
 $w(0)>0$,\quad  $w'(0)=0$. 
\end{tabular}\right.
\end{equation}
with $\theta\geq0$, $\Lambda\in\R\smallsetminus\{0\}$ and $p>1$, which includes \eqref{Eq:MainReduction A=0}, \eqref{Eq:Main A neq 0 gamma=0} and \eqref{Eq:Main Aneq0 gamma neq 0 general} as particular cases.

We have a Pohozaev type identity for the solutions to this equation

\begin{lemma}\label{Lemma:Pohozaev}
If $w\in C^2([0,\infty))$ is a solution to problem \eqref{Eq:GeneralEmdenFowlerMinkowski} with $\Lambda>0$, then for every $r>0$ the following identity holds true
\[
-r^{\theta+1}E_w(r) - \frac{\theta - 1}{2} r^{\theta}w(r)w'(r)=\Lambda\left[\frac{\theta-1}{2}-\frac{\theta+1}{p+1}\right]\int_0^r s^{\theta}\vert w(s)\vert^{p+1}ds
\]
where $E_w$ denotes the energy of the solution $w$, given by
\[
E_w(r):=\frac{1}{2}\vert w'(r)\vert^2+\frac{\Lambda}{p+1}\vert w(r)\vert^{p+1}.
\]
\end{lemma}

\begin{proof}
It is a straightforward application of the Pucci-Serrin's variational identity (see \cite[Proposition 1]{PuSe}), applied to the one dimensional Lagrangian $\mathcal{F}(s,u,q):=s^\theta\left[ \frac{1}{2}q^2 -  \frac{\Lambda}{p+1}\vert u\vert^{p+1}\right]$, considering the functions $h(r):=r$ and $a(r)\equiv\frac{\theta -1}{2}$.
\end{proof}

A solution to problem \eqref{Eq:GeneralEmdenFowlerMinkowski} will be called \emph{proper} if it is defined for all $r>0$. On the other hand, if there exists $R<\infty$ such that $\vert w(r)\vert\rightarrow\infty$ as $r\rightarrow R$, we will say that the solution \emph{blows up in finite time}. We say that a proper solution is \emph{stable} if $\lim_{r\rightarrow\infty}w(r)=0$, otherwise it is called \emph{not stable}. Finally, we say that a proper solution is \emph{oscillatory} if $w$ possesses arbitrarily large zeroes, that is, given any $M>0$, there exists $r_0>M$ such that $w(r_0)=0$. We have the following theorem

\begin{theorem}\label{Th:EmdenFowlerMinkowski}
Let $\Lambda\in\R\smallsetminus\{0\}$, $\theta\geq 0$ and $p>1$. 
\begin{itemize}
	\item If $\Lambda<0$, all the solutions to \eqref{Eq:GeneralEmdenFowlerMinkowski} are positive, strictly increasing and blow up in finite time.
	\item If $\Lambda>0$, then the solutions to \eqref{Eq:GeneralEmdenFowlerMinkowski} are proper with $w$ and $w'$ uniformly bounded in $[0,\infty)$. Moreover,
	\begin{itemize}
		\item if $\theta=0$, all the solutions are oscillatory and not stable;
		\item if $0<\theta<\frac{p+3}{p-1},$ all the solutions are oscillatory and stable;
		\item if $\theta\geq \frac{p+3}{p-1},$	all the solutions are positive, monotone decreasing and stable.
	\end{itemize}
\end{itemize}
\end{theorem}

\begin{proof}
The case of $\Lambda<0$ follows from Proposition \ref{Prop:Blow-Up}, for $\theta\neq0$ and a slight modification of the proof of the same proposition yields the blow-up for $\theta=0$.

Now we focus in the case $\Lambda>0$. In this case, Proposition \ref{Prop:Global existence} gives the global existence and boundness of the solutions. By a well known result \cite[Corollary 2.1]{WiWo}, for every $\theta>0$, the solution is stable.
By Theorem 3.3 in \cite{FePe}, the solutions are oscillatory if $0\leq \theta<\frac{p+3}{p-1}$. 

Now we take $\theta\geq \frac{p+3}{p-1}$, or equivalently $\frac{\theta-1}{2}\geq \frac{\theta+1}{p+1}$. Suppose, in order to get a contradiction that the solution $w$ have a zero $r_0>0$. By uniqueness of the solutions, $w'(r_0)\neq0$, otherwise $w\equiv 0$, which is a contradiction to the fact that $w(0)>0$. Evaluating the identity in Lemma \ref{Lemma:Pohozaev} at $r=r_0$, we conclude that
\[
0>-\frac{r_0^{\theta+1}}{2}\vert w'(r_0)\vert^2 =\Lambda\left[\frac{\theta-1}{2}-\frac{\theta+1}{p+1}\right]\int_0^{r_0} s^{\theta}\vert w(s)\vert^{p+1}ds\geq 0,
\]
a contradiction since $\Lambda>0$ and $\frac{\theta-1}{2} - \frac{\theta+1}{p+1}\geq0$. Hence the solution must be positive if $\theta\geq \frac{p+3}{p-1}$. By the positivity of the solutions, a similar argument to the one given in the proof of Lemma \ref{Lemma:ww' Positive} yields that $w'<0$ for every $r>0$, so that $w$ is monotone decreasing. 

Finally, to prove that the solution is not stable when $\theta=0$, observe that the energy is constant in this situation, for $E'_w(r)=0$. Hence, if $w(r)\rightarrow 0$ as $r\rightarrow\infty$, then $0<E_w(0)=\lim_{r\rightarrow\infty}E(r)=\frac{1}{2}\lim_{r\rightarrow\infty}\vert w'(r)\vert$. Thus, there exists $r_\ast>0$ such that $\vert w'(r)\vert >0$ for every $r\geq r_\ast$ and the solution cannot oscillate, which is a contradiction. 
\end{proof}

\subsection{Qualitative properties of the solutions to the ODE's \eqref{Eq:MainReducedYamabeMinkowski} and \eqref{Eq:CartanMunznerEntireODE}, and proof of Theorem \eqref{Th:NodalDeSitterIsop}}\label{Sec:Proof 2}

We will apply all the theory developed in this section to show the existence of global and blowing-up solutions to equations \eqref{Eq:MainReducedYamabeMinkowski} and \eqref{Eq:CartanMunznerEntireODE} and to describe their qualitative behaviors.

\subsubsection{The case of the semi-Riemannian Euclidean space.}

As we have seen in Section \ref{Section:ReductionMinkowski} and in Appendix \ref{App:Minkowski}, the isoparametric functions considered in Example \ref{Ex:IsoparametricMinkowsky} fall into one of the following four categories

\begin{itemize}
\item[(M1)] Im$\,\varphi=\R$ and $\varphi$ has no critical points,
\item[(M2)] Im$\,\varphi=[0,\infty)$ or $(-\infty,0]$, and $t=0$ is the unique critical value of $\varphi$,
\item[(M3)] Im$\,\varphi=\R$ and $t=0$ is the unique critical value of $\varphi$.
\end{itemize}

According to this, we have the following result for equation \eqref{Eq:MainReducedYamabeMinkowski}.

\begin{theorem}\label{Th:SolutionsMinkowskiODE}
Let $\varphi:\R^m_s\rightarrow\R$ be isoparametric of one of the above types. Then equation \eqref{Eq:MainReducedYamabeMinkowski} admit the following types of solutions
\begin{itemize}[leftmargin=*]
	\item For $\varphi$ of type (M1),
		\begin{itemize}
			\item[(M1.1)] If $A=0$, $\mu\neq0$ and $p>1$, or if $A\neq0$ with $\beta=0$, $\mu>0$ and $p>1$, the solution is globally defined in $\R$, is bounded, oscillatory and not stable.
			\item[(M1.2)] If $A=0$, $\mu\neq 0$ and $p>1$ or if $A\neq0$ with $\beta=0$, $\mu<0$ and $p>1$, the solution is defined in some interval of the form $(-R,R)$ with $0<R<\infty$, is even, monotone in $(0,R)$ and blows-up in finite as $t\rightarrow\pm R$, with a global minimum at $t=0$.
		\end{itemize}
		\item For $\varphi$ of type (M2) with $\delta=0$ and Im$\,\varphi=[0,\infty)$, 
		\begin{itemize}
			\item[(M2.1)] If $\mu\gamma>0$, $\frac{2\beta}{\gamma}=1$ and $p>1$, then the solution is globally defined in $[0,\infty)$, is bounded, oscillatory and not stable as $t\rightarrow \infty$.
			\item[(M2.2)] If $\mu\gamma>0$  and $p<\frac{\beta+\gamma}{\beta-\gamma}\leq\infty$, the solution is globally defined in $[0,\infty)$, is bounded, oscillatory and stable as $t\rightarrow\infty$, with a global maximum at $t=0$.
			\item[(M2.3)] If $\mu\gamma>0$ and $p\geq\frac{\beta+\gamma}{\beta-\gamma}$, the solution is globally defined in $[0,\infty)$, is bounded, positive, monotone decreasing and stable as $t\rightarrow\infty$, being $t=0$ the unique critical point and it is a global maximum..
			\item[(M2.4)] If $\mu\gamma<0$ and $p>1$, then the solution is defined for in some interval of the form $[0,R)$ with $0<R<\infty$, is monotone increasing and blows-up as $t\rightarrow R$,  being $t=0$ the unique critical point and it is a global minimum.
		\end{itemize}
	The same behavior holds true if Im$\,\varphi=(-\infty,0]$, reversing the inequalities for $\mu\gamma$ and taking $t\rightarrow-\infty$ instead of $t\rightarrow\infty$ in (M2.1)-(M2.4).
\item For $\varphi$ of type (M3)  with $\delta=0$, 
		\begin{itemize}
			\item[(M3.1)] If  $\mu\gamma<0$ and $p<\frac{\beta+\gamma}{\beta-\gamma}$, then the solution is defined in $(-\infty,R)$ for some $R>0$; in $(-\infty,0]$ is bounded, oscillatory and stable as $t\rightarrow-\infty$, while in $(0,R)$ is monotone increasing and blows-up as $t\rightarrow R$.
			\item[(M3.2)] If $\mu\gamma<0$ and $p\geq\frac{\beta+\gamma}{\beta-\gamma}$, then the solution is positive and defined in $(-\infty,R)$ for some $R>0$; in $(-\infty,0]$ is bounded, monotone increasing and stable as $t\rightarrow-\infty$, while in $(0,R)$ is monotone increasing and blows-up as $t\rightarrow R$, being $t=0$ the unique critical point and it is an inflection point.
		\end{itemize}
	A similar behavior is true if we reverse the inequality for $\mu\gamma$ and change $(-\infty,R)$, $(-\infty,0]$ and $(0,R)$ by $(R,\infty)$, $[0,\infty)$ and $(-R,0)$, respectively, in statements (M3.1) and (M3.2). 	
\end{itemize}
\end{theorem}

\begin{proof}
The proof is an easy consequence of Theorem \ref{Th:EmdenFowlerMinkowski}, but, for the reader convenience, we make the proof for one of the above cases. For instance, we next show the existence of a solution of type (M3.1) when $\varphi$ is of type (M3), $\mu\gamma<0$ and $p<\frac{\beta+\gamma}{\beta-\gamma}$.  In this situation,  $2\beta=\gamma\text{tr}A=\gamma(k+n)$, $\delta=0$ and \eqref{Eq:MainReducedYamabeMinkowski} is defined in $\R$. By the analysis made in Section \ref{Section:ReductionMinkowski}, a solution to \eqref{Eq:MainReducedYamabeMinkowski} with initial conditions $v(0)=d>0$ and $v'(0)$ is given by 
\[
v(t)=\left\{\begin{tabular}{cc}
$w_{\Lambda_1}(\sqrt{t})$ & if $t\geq0$\\
$w_{\Lambda_2}(\sqrt{-t})$ & if $t\leq0$
\end{tabular}\right.
\]
where $w_{\Lambda_1}$ is a solution to problem \eqref{Eq:GeneralEmdenFowlerMinkowski} with $\theta=k+n-1$, $\Lambda_1=\frac{4\mu}{\gamma}<0$ and initial conditions $w_{\Lambda_1}(0)=d$ and $w'_{\Lambda_1}(0)=0$, while $w_{\Lambda_2}$ is a solution to problem \eqref{Eq:GeneralEmdenFowlerMinkowski} with $\theta=k+n-1$, $\Lambda_2=-\frac{4\mu}{\gamma}>0$ and initial conditions $w_{\Lambda_2}(0)=d$ and $w'_{\Lambda_2}(0)=0$. As $v(0)=w_{\Lambda_1}(0)=w_{\Lambda_2}(0)=d>0$ and $v'(0)=w'_{\Lambda_1}(0)=w'_{\Lambda_2}(0)=0$, by existence and uniqueness of the solutions, $v$ is a well defined solution to \eqref{Eq:MainReducedYamabeMinkowski}. Now, on the one hand, as $\Lambda_1<0$, by Theorem \ref{Th:EmdenFowlerMinkowski}, $w_{\Lambda_1}$ blows up in finite time $R'>0$, is positive and monotone increasing in $(0,R)$. On the other hand, taking $\theta=k+n-1$, inequality $p<\frac{\beta+\gamma}{\beta - \gamma}=\frac{(k+n)+2}{(k+n)-2}=\frac{\theta+3}{\theta-1}$ is equivalent to $\theta<\frac{p+3}{p-1}$; hence Theorem \ref{Th:EmdenFowlerMinkowski} yields that $w_{\Lambda_2}$ is globally defined, bounded, oscillates and is stable as $r=\sqrt{-t}\rightarrow\infty$. Thus, $v$ blows up as $t\rightarrow R:=R'^2$, is positive and monotone increasing in $(0,R)$, while it is defined for every $t\leq0$, is bounded, oscillates in $(-\infty,0)$ and it is stable as $t\rightarrow-\infty$, which are the characteristics of the solution (M3.1). 

Similar arguments yield the existence of the other type of solutions.
\end{proof}

The main theorems for the Yamabe equation on the Minkowski space follow easily. 

\begin{proof}[Proof of Theorem \ref{Th:MinkowskiIsop}]
	It follows from Proposition \ref{Prop:ReductionODE} and Theorem \ref{Th:SolutionsMinkowskiODE}.
\end{proof}

\begin{proof}[Proof of Theorem \ref{Th:GeometrySolutions}]
This follow directly from Theorem \ref{Th:SolutionsMinkowskiODE} and the analysis in Appendix \ref{App:Minkowski}, because $u^{-1}(c)=\varphi^{-1}(v^{-1}(c))$ and $\nabla u=v'(\varphi)\nabla\varphi=0$ if and only if $\nabla \varphi=0$ or $v'(\varphi)=0$. 
\end{proof}


\subsubsection{The case of the pseudosphere.}

Now we turn our attention to equation \eqref{Eq:CartanMunznerEntireODE}. We will consider isoparametric functions given by the restrictions of semi-Riemannian Cartan-M\"{u}nzner polynomials. As it was mentioned in Section \ref{Section:ReductionDeSitter}, the possible images of such functions fall into one of the following categories
\begin{itemize}
	\item[(P1)] Im$\,\varphi=\R$;
	\item[(P2)] Im$\,\varphi=[-1,\infty)$ or Im$\,\varphi=(-\infty,1]$;
	\item[(P3)] Im$\,\varphi=[1,\infty)$ or Im$\,\varphi=(-\infty,-1]$.
\end{itemize}
In all these cases, the only critical values of $\varphi$ are $t=\pm1$. In Appendix \ref{App:DeSitter} we see that the functions in Example \ref{Ex:IsoparametricDeSitter} may have all the possible images in (P1)-(P3).  Recall the definition of $n_1,n_2$ and $\kappa$ in Theorems \ref{Th:NodalDeSitterIsop} and \ref{Th: k zeroes}. According to this, we have the following existence and multiplicity result for equation \eqref{Eq:CartanMunznerEntireODE}.

\begin{theorem}\label{Th:SolutionsDeSitterEntireODE}
Let $\varphi:\S_s^m\rightarrow\R$ be an isoparametric function given by the restriction of a Cartan-M\"{u}nzer polynomial of type (P1), (P2) or (P3), and let $k\geq 1$ be a positive integer. Suppose, in case of (P2) and (P3), that the images are $[-1,\infty)$ and $[1,\infty)$ respectively. If $p<\frac{(m-\kappa)+2}{(m-\kappa)-2}$, then equation \eqref{Eq:CartanMunznerEntireODE}, with $v(1)>1$ and the natural boundary condition \eqref{Eq:NaturalBoundaryConditions} at $t=1$, admits the following two types of blowing up solutions.
\begin{itemize}
	\item[(P1.k)] $v$ is sign-changing and defined in $[-1,R_+)$ for some $R_+>1$, has exactly $k$ zeroes, all in $(-1,1)$, $v$ is positive, monotone increasing in $(0,R_+)$ and $v(t)\rightarrow\infty$ as $t\rightarrow R_+$.
	\item[(P2.k)] $v$ is sign-changing and defined in $(-R_-,R_+)$ for some $R_-,R_+>1$, has exactly $k$ zeroes, all in $(-1,1)$, $v$ is positive, monotone increasing in $(1,R_+)$ and $v(t)\rightarrow\infty$ as $t\rightarrow R_+$, while $v$ is negative, monotone in $(-R_-,-1)$ and $v(t)\rightarrow-\infty$ as $t\rightarrow -R_-$ if $k$ is odd, or it is positive, monotone and $v(t)\rightarrow\infty$ as $t\rightarrow-R_-$ if $k$ is even.
	\item[(P3.1)] $v$ is positive, defined in $[1,R_+)$ for some $R_+>1$, monotone increasing with a global minimum at $t=1$ and $v(t)\rightarrow\infty$ as $t\rightarrow R_+$.
\end{itemize}
A similar result holds true if we consider the image of (P2) and (P3) to be $(-\infty,1]$ and $(-\infty,-1]$, respectively, taking the initial condition at $t=-1$ instead that $t=1$ for (P3).
\end{theorem}

\begin{proof}
We just show the existence of type (P2.k)  solutions for $k$ odd, being the other similar and simpler. In this case we are considering that Im$\,\varphi=\R$ and that the only critical values of $\varphi$ are $t=\pm1$. 

First, as the function $f(t)=\vert t\vert^{p-1}t-t$ is locally Lipschitz continuous, a standard contraction mapping argument yields the existence and uniqueness of local solutions to the problem \eqref{Eq:CartanMunznerEntireODE} with initial conditions $v(t_0),v'(t_0)\in\R$ for $t_0\neq\pm 1$, and natural initial conditions $v(\pm1)\in\R$ and
\begin{equation}\label{Eq:NaturalInitialCondition}
v'(\pm1)= -\frac{\lambda\left[ \vert v(\pm1)\vert^{p-1}v(\pm1)-v(\pm1)\right]}{\ell^2(-\frac{m-1+\ell}{\ell}(\pm 1)+\beta)}
\end{equation} 
for $t_0=\pm 1$ (see Appendix \ref{App:ExistenceUniquenessSingular}).

Now, as it was shown in Section \ref{Section:ReductionDeSitter}, solving problem \eqref{Eq:CartanMunznerEntireODE} in $\R$ reduces to solving it separately in $(-\infty,-1]$, $[-1,1]$ and $[1,\infty)$ with suitable boundary conditions so that the solutions can be glued together using the local existence and uniqueness to the problem. Take $d:=v(1)>1$ and the natural initial condition $v(1)$ given by \eqref{Eq:NaturalInitialCondition}. By the subcritical condition for $p$ and Theorem \ref{Th: k zeroes}, there exists a  solution $w_k$ to problem \eqref{Eq:Compact} with initial conditions $w_k(0)=d>0$, $w'(0)=0$ and having exactly $k$ zeroes in $(0,\pi)$. On the other hand, Proposition \ref{Prop:Blow-Up} gives a blowing up solution $w_+$ to the problem \eqref{Eq:PositiveODE} with initial conditions $w_+(0)=d$ and $w_+'(0)=0$, which is positive, monotone increasing and blows up as $r\rightarrow R'_+$ for some $R_+'>0$.
As $k$ is odd by hypothesis,  $-c:=w_k(\pi)<-1$ and $w'(\pi)=0$ for some $c>1$ by Theorem \ref{Th: k zeroes}. Let $w_-$ be a solution to problem \eqref{Eq:PositiveODE} with $h_-$ and initial conditions $w_-(0)=c>1$ and $w_-'(0)=0$. By Proposition \ref{Prop:Blow-Up}, this solution is positive, monotone increasing and blows up as $r\rightarrow R_-'$ for some $R_->0$. Notice that $-w_-$ is also a solution, but now with $-w_-(0)=-c=w_k(\pi)$, and it is negative, monotone and blows up as $r\rightarrow R_-'$.

Hence, by the reductions made in \ref{Section:ReductionDeSitter}, defining $R_\pm:=\text{arccosh}(R_\pm')$, the function 
\[
v(t):=\left\{\begin{tabular}{cc}
$-w_-(\text{arccosh}\,(-t)),$ & if $-R_-<t\leq -1$\\
$w_k(\arccos t),$ & if $-1\leq t \leq 1$\\
$w_+(\text{arccosh}\, t),$ & if $1\leq t < R_+$
\end{tabular}\right.
\] 
is a well defined solution to \eqref{Eq:CartanMunznerEntireODE} in $(-R_,R_+)$ having the desired properties. 
\end{proof}

The main result for the Yamabe type problem on the pseudosphere is a simple consequence.

\begin{proof}[Proof of Theorem \ref{Th:NodalDeSitterIsop}]
The proof follows from Proposition \ref{Prop:ReductionODE} together with Theorem \ref{Th:SolutionsDeSitterEntireODE}.
\end{proof}

We make a final remark about the solutions to the Yamabe type equation on the pseudosphere. In Theorem \ref{Th:SolutionsDeSitterEntireODE}, we did not obtain globally defined solutions to the Yamabe-type problem \eqref{Eq:MainYamabeSphereNorm}. This leads us to the following

\begin{question}
Does equation \eqref{Eq:CartanMunznerEntireODE} admit global solutions different from $v\equiv 0,1$? If this is the case, are there positive solutions?
\end{question}

As we have seen in Theorem \ref{Th:SolutionsDeSitterEntireODE}, this is equivalent to study the global behavior of the solutions to equation \eqref{Eq:PositiveODE} with initial conditions $w(0)\in(0,1)$ and $w'(0)=0$ and to prove the existence of solutions to \eqref{Eq:Compact} having initial and final values $\vert v(\pm 1)\vert <1$. In case of the equation \eqref{Eq:PositiveODE}, neither Proposition \ref{Prop:Global existence} nor \ref{Prop:Blow-Up} apply to prove the global existence or the blow-up of the solutions. Concerning equation \eqref{Eq:Compact}, as it was mentioned immediately after Theorem \ref{Th: k zeroes}, it is not clear what the initial and final conditions for the positive solutions to this problem are; if they are both bigger that one, our result would yield the existence of positive blowing up solutions. However, if both lie in $(0,1)$, it may be possible that a globally defined solution to \eqref{Eq:CartanMunznerEntireODE} exists.


\appendix

\section{Some global isoparametric functions in the semi-Riemannian Euclidean space form}\label{App:Minkowski}

Let $\varphi:\R^m_s\rightarrow\R$ be the isoparametric function given by $\varphi(z) =\langle Az, z \rangle + 2\langle a, z \rangle$ where $A\in$Sym$(\R^m_s)$ and $a\in\R^m_s$ and $\alpha\in\R$ satisfy that $(A-\alpha I)A=0$ and $Aa=\alpha a$. We will study the level sets of these functions according to the following three cases: $A=0$, $A\neq0$ with tr$A\neq0$ and $A\neq0$ with tr$A=0$. In the second case we will suppose, for simplicity, that $a=0$ and in the third case we will give an explicit example generalizing the one given in \cite{Ha}, but other possibilities could occur in both cases.

\subsection{Case $A=0$.} The isoparametric function is just $\varphi(z) = 2\langle a, z \rangle$ and is not constant if and only if $a\neq 0$. Here Im$\varphi=\R$, every $c\in\R$ is a regular value and 
\[
\varphi^{-1}(c)=\mathcal{L}^{m-1}_a(c/2):=\{z\in\R^m \;:\; \langle a,z \rangle = c/2 \},
\]
i.e., the isoparametric hypersurfaces are parallel hyperplanes.

\subsection{Case $A\neq0$, $a=0$ and $\alpha\neq 0$.} The conditions on $A$ imply that the minimal polynomial of $A$ is $t, t-\alpha$ or $(t-\alpha)t$, where we obtain that $A$ is diagonalizable with proper values $\alpha\neq 0$ and $0$, and proper vector $a$ associated to $\alpha$. Hence, we can suppose, without loss of generality, that $A$ has the form
\[
A=\text{diag}(
\overbrace{\alpha,\dots,\alpha}^k,
\overbrace{0,\dots,0}^{s-k},
\overbrace{\alpha,\dots,\alpha}^n,
\overbrace{0,\dots,0}^{m-s-n}
);
\]
where $0\leq k\leq s$ and $0\leq n\leq m-s$. As $A\neq 0$, we have that $k\neq 0$ or $n\neq 0$.
Hence $\varphi(\bar{t},\bar{x})=\alpha\left[-\sum_{i=1}^kt_i^2 + \sum_{j=1}^n x_j^2\right]$ and
\[
\text{Im}\varphi=
\left\{\begin{tabular}{cc}
$(-\infty,0]$ & if $k=0$, $n\neq0$ and $\alpha<0$, or if $n=0$, $k\neq0$ and $\alpha>0$,\\
$[0,\infty)$ & if $k=0$, $n\neq0$ and $\alpha>0$, or if $n=0$, $k\neq0$ and $\alpha<0$,\\
$\R$ & if $k\neq0$ and $n\neq 0$.
\end{tabular}\right.
\]
and the only critical value of $\varphi$ is $t=0$.
Therefore, the possible level sets are given as follows:
\begin{itemize}
	\item If $k=0$ and $n\neq 0$,

	\[
	\varphi^{-1}(c)=\left\{\begin{tabular}{cc}
	$\emptyset$   &   if $\frac{c}{\alpha}<0$\\
	$\{0\}\times\R_s^{m-n}$   &   if $c=0$\\
	$\S^{n-1}(\frac{c}{a})\times\R^{m-n}_s$   &  if $\frac{c}{\alpha}>0$.
	\end{tabular}\right.
	\]
	
	\item If $n=0$ and $k\neq 0$,
	
	\[
	\varphi^{-1}(c)=\left\{\begin{tabular}{cc}
	$\emptyset$   &   if $\frac{c}{\alpha}>0$\\
	$\R_{s-k}^{m-k}\times\{0\}$   &   if $c=0$\\
	$\S^{k-1}(-\frac{c}{a})\times\R^{m-k}_{s-k}$   &  if $\frac{c}{\alpha}<0$.
	\end{tabular}\right.
	\]
	
	\item If $2\leq k+n<m$ and $k,n\neq0$,
	
	\[
	\varphi^{-1}(c)=\left\{\begin{tabular}{cc}
	$\S_k^{(k+n)-1}\times\R_{s-k}^{m-(k+n)}$   &   if $\frac{c}{\alpha}>0$\\
	$\mathcal{C}_k^{(k+n)-1}\times\R_{s-k}^{m-(k+n)}$   &   if $c=0$\\
	$\H_{k-1}^{(k+n)-1}(-\frac{c}{a})\times\R_{s-k}^{m-(k+n)}$   &  if $\frac{c}{\alpha}<0$.
	\end{tabular}\right.
	\]
	
	\item If $k+n=m$, i.e., if $k=s$ and $n=m-s$,

	\[
	\varphi^{-1}(c)=\left\{\begin{tabular}{cc}
	$\S_s^{m-1}$   &   if $\frac{c}{\alpha}>0$\\
	$\mathcal{C}_s^{(m)-1}$   &   if $c=0$\\
	$\H_{s-1}^{m-1}(-\frac{c}{a})$   &  if $\frac{c}{\alpha}<0$.
	\end{tabular}\right.
	\]
	
\end{itemize}

\subsection{Example of an isoparametric function with $A\neq0$, $A^2=0$ and $\alpha=0$.}

In order to simplify the ideas, let $s<m-s$, so that we can write the signs of the metric in $\R^{m}_s$ in pairs as follows
\[
(\underbrace{-,+,-,+,\dots,-,+}_{2s},\underbrace{+,\ldots,+}_{m-2s}) 
\]
that is, we write first $s$ pairs of $(-,+)$ signs and we let $m-2s$ ``$+$" signs at the end. Accordingly, we write $z\in\R_s^m$ as $z=(t_1,x_1,\ldots,t_s,x_s,x_{s+1},\ldots,x_{m-s})$. Observe that the symmetric operators in $\R^2_1$ must have the following matrix representation in orthonormal basis. 
\begin{equation*}
c
\begin{pmatrix}
-1 & \pm 1 \\
\mp 1 & 1
\end{pmatrix},
\end{equation*}
with $c\in\R$. 

Hence, we are led to consider only the matrices $A_+$ and $A_-$ defined by 
\begin{equation*}
A_+=
\begin{pmatrix}
-1 & 1 \\
-1 & 1
\end{pmatrix}
\quad\text{and}\quad
A_-=\begin{pmatrix}
-1 & -1 \\
1 & 1
\end{pmatrix}.
\end{equation*}
Observe that the eigenvectors of these operators are {\em null vectors}. 
For dimensions $m\geq 3$, consider the matrix
\begin{equation*}
A=
\begin{pmatrix}
A_{1} & 0 & \cdots & 0  & 0 \\
0 & A_{2} & \cdots & 0 & 0 \\
\vdots & \vdots & \ddots & \vdots & \vdots \\
0 & 0 & \cdots & A_{s}  &  0 \\
0 & 0 & \cdots & 0 & \bar{0}_{m-2s}
\end{pmatrix}
\end{equation*}
where $A_{i}$ is either $A_+$ or $A_-$ and $\bar{0}_{m-2s}$ is the zero matrix in dimension $m-2s$. This matrix belongs to Sym$(\R^m_s)$, $\alpha=0$ is the only eigenvalue and satisfies that $A^2=0$, tr$A=\sum_{i=1}^s$tr$A_{i_k}=0$ and that for any $1\leq\omega\leq m-2s$, the vector
\[
a=(\underbrace{0,\ldots,0}_{2s},\underbrace{1,1\ldots,1}_{\omega},\underbrace{0,\ldots,0}_{(m-2s)-\omega})
\]
is an eigenvector with eigenvalue $\alpha=0$. Let $\epsilon_i=+1$ if $A_i=A_+$ and $\epsilon_{i}=-1$ if $A_i=A_-$. Hence
\begin{equation*}
\langle Az,z \rangle = \sum_{i=1}^s (t_i-\epsilon_i x_i)^2
\end{equation*}

Then, for any $1\leq\omega\leq m-2s$, the function
\begin{equation*}
\varphi(z)=\langle Az,z\rangle + 2\langle a,z \rangle = \sum_{i=1}^s (t_i-\epsilon_i x_i)^2 + 2\sum_{j=s+1}^{\omega}x_j
\end{equation*}
is isoparametric, Im$\varphi=\R$. For any $\nu,\omega\in\N$, let  $\epsilon_1,\epsilon_s$. Then we define the parabolic cylinder of dimension $k-1:=2s+\omega-1$ and radius $c$ as
\[
\mathcal{P}^{k-1}(c):=\mathcal{P}^{k-1}(\{\epsilon_i\}_{i=1}^{\nu},\omega):=\{(\bar{t},\bar{x})\in\R_\nu^k \;:\; \sum_{i=1}^\nu (t_i-\epsilon_i x_i)^2 + 2\sum_{j=\nu+1}^{\omega}x_j=  c \}.
\]  
With this notation, the level sets of $\varphi$ are given by 
\[
\varphi^{-1}(c)=\left\{\begin{tabular}{cc}
$\mathcal{P}^{2s+\omega-1}\times\R^{m-2s-\omega}$ & if $\omega < m-2s$\\
$\mathcal{P}^{m-1}$ & if $\omega = m-2s$,
\end{tabular}\right.
\]
i.e., they are all parallel parabolic cylinders.

\section{Some remarks about global isoparametric functions in the De Sitter space form}\label{App:DeSitter}
 
Let $\varphi:\S_s^m\rightarrow\R$ be an isoparametric function given in Example \ref{Ex:IsoparametricDeSitter}. In what follows, $M_c:=\varphi^{-1}(c)$ for $c\neq\pm$ and $M_{\pm}=\varphi^{-1}(\pm1)$. We will next sketch the main properties of these functions and describe the possible isoparametric hypersurfaces $M_c$ and focal varieties $M_\pm$ they define.

\subsection{Linear examples.}

For $Q\in \S_s^m$, the function $\Phi(z):=\langle Q,z \rangle$ is a semi-Riemannian Cartan-M\"{u}nzner polynomial in $\R_s^{m+1}$ of degree $\ell=1$ and $\beta=0$, therefore $\varphi:=\Phi\upharpoonright_{\S_s^m}$ is isoparametric. As any hyperplane intersects $\S_s^m$, Im$\,\varphi=\R$ and the only critical values of this function are $c=\pm 1$. The isoparametric hypersurfaces, which correspond to the regular level sets of $\varphi$ are given by the intersection of the planes $\mathcal{L}_{Q,c}:=\{z\in\R_s^{m+1}\;:\;\langle Q,z \rangle = c\}$ and $\S_s^{m}$. Actually, this hypersurfaces have only $\ell=1$ principal curvature of multiplicity $m_1=m_2=m-1$. Hence, the classification theorem for proper hypersurfaces in $\S_s^m$ \cite{AbKoYa} yields that, up to isometries, they must be homothetic to $\S_s^{m-1}$, to $\S_{s-1}^{m-1}$ or to $\H_{s-1}^{m-1}$. On the other hand, the focal varieties $M_{\pm}:=\varphi^{-1}(\pm)$ may not be submanifolds of $\S_s^{m}$, but algebraic varieties. In order to get a better understanding of the hypersurfaces and focal varieties in the linear example, suppose that $Q=e_{n+1}$, then if $M_c:=\varphi^{-1}(c)$, we have that
\[
M_c=\left\{\begin{tabular}{cc}
$\S_s^{m-1}$ & if $-1<c<1$\\
$\H_{s-1}^{m-1}$ & if $\vert c\vert > 1$,
\end{tabular}\right.,
\]
while the focal varieties are the null cones
\[
M_{\pm}:=\mathcal{C}^{m-1}_s\times\{\pm 1\}.
\]
Notice that neither $M_-$ nor $M_+$ are submanifolds of $\S_s^m$ and that their dimension as algebraic varieties is $m-1\neq 0= n_i$, where $n_i := (m-1)- m_i$ for $i=1,2$.

\subsection{Quadratic examples.} For simplicity, we will only consider the case in which $A\in\text{Sym}(\R_s^m)$ has a minimal polynomial of the form $t^2+a t + b$ with $a,b\in\R$ satisfying $a^2-4b>0$. This implies that the shape operator is diagonalizable with exactly $\ell=2$ distinct real principal curvatures, with corresponding multiplicities $m_1$ and $m_2$. Again, by the classification theorem for proper  hypersurfaces in semi-Riemannian real space forms \cite{AbKoYa}, the isoparametric hypersurfaces are congruent to $\S_s^{r}(c_1)\times\S_{s-r}^{m-1-r}(c_2)$ or to $\S_r^{r}(c_1)\times\H^{m-1-r}_{s-r-1}$ for some $r<m-1$ and $c_1,c_2\in\R$ satisfying $c_1+c_2=1$ with $c_1,c_2>0$ in the first case and $c_1>0, c_2<0$ in the second. 

We fix now a concrete example which is given by a Cartan-M\"{u}nzner polynomial. Fix $k_1>\max{s,2}$ and $k_2>2$ such that $k_1+k_2=m+1$ and consider $A\in$Sym$\R^{m+1}_s$ given by the diagonal matrix
\[
A=\text{diag}(\underbrace{-1,\ldots,-1}_{k_1},\underbrace{1,\ldots,1}_{k_2}).
\]
Then, the polynomial $\Phi(z):=\langle Az,z \rangle$ is a Cartan-M\"{u}nzner polynomial of degree $\ell=2$ and $\beta=\frac{k_2-k_1}{2}$, implying that $\varphi:=\Phi\upharpoonright_{\S_s^m}$ is isoparametric. A straightforward computation yields that Im$\,\varphi=[1,\infty)$ and that $c=\pm1$ are the only critical values of $\varphi$. Since the minimal polynomial of $A$ is $t^2-1$, then the previous remarks yield that the isoparametric hypersurfaces $M_c:=\varphi^{-1}(c)$ have $\ell=2$ distinct principal curvatures with multiplicities $m_1=(m-1)-(k_1-1)=k_2-1$ and $m_2=(m-1)-(k_2-1)=k_1-1$. Notice that $\frac{m_1+m_2}{2}=\frac{k_1+k_2-2}{2}=\frac{m-1}{\ell}$. If $c=\cosh 2t$ with $t>0$, recalling that $\cosh^2 t+\sinh^2 = \cosh 2t$ and that $\cosh^2t-\sinh^2t=1$, the isoparametric hypersurface $M_c$ is isometric to $\H^{k_1-1}_{s-1}(\sinh t)\times\S^{k_2-1}(\cosh t)$. For $-11<c<1$, taking $c=\cos 2t$ for some $t\in(0,\pi/2)$, then $M_c$ is isometric to $\S^{k_1-1}_s(\sqrt{1-\cos 2t})\times\S^{k_2-1}(\sqrt{1+\cos2t})$. Notice that $n_1:=(m-1)-m_1=k_1-1$ and $n_2=(m-1)-m_2=k_2-1$ are the dimensions of each factor. The focal varieties are $M_+=\mathcal{C}_s^{k-1}\times\S^{k_2-1}(\sqrt{2})$ and $M_-=\S^{k_1-1}_s(\sqrt{2})\times\{0\}$. Here it is evident that the focal variety $M_+$ is not a submanifold of $\S_s^m$ and that its dimension does not coincide with $n_1=k_1-1$, while $M_-$ is a nondegenerate submanifold of dimension $n_2=k_2-1$.

\subsection{Clifford examples.} Let $m=2k-1$ and consider a Clifford system $(P_1,\ldots,P_n)$ of signature $(n,r)$. Define $m_1:=n-1$ and $m_2:=k-n$, and let $\eta_j=-1$ if $j\leq r$ and $\eta_j=1$ if $r<j$.  Then the function $\Phi(z):=\langle z,z \rangle^2 - 2\sum_{j=1}^{n}\eta_{j}\langle P_jz,z \rangle^2$ is a Cartan-M\"{u}nzner polynomial of degree $\ell=4$ and $\beta=\frac{m_2-m_1}{2}=\frac{k+1-2n}{2}$, see \cite{Ha}. Also notice that $\frac{m_1+m_2}{2}=\frac{k-1}{2}=\frac{m-1}{\ell}$. Depending on $r$ and $n$, the image of the isoparametric function $\varphi=\Phi\upharpoonright_{\S_s^m}$ may be $\R$, $(-\infty,-1]$, $(-\infty,1]$, $[-1,\infty)$ and $[1,\infty)$, and all these options are possible \cite{Ha1}. The only critical values of $\varphi$ are $c=\pm 1$ and it can be proved (see \cite{Ha}) that the focal varieties $M_\pm:=\varphi^{-1}(\pm1)$ are nondegenerate submanifolds of $\S_s^m$. Therefore, we can consider the normal space at $z\in M_\pm$, which we denote by $N_zM_\pm$. Set
\[
BM_{\pm}(\delta):=\{(z,V)\;:\; z\in M_\pm, V\in NM_z, \langle V,V \rangle=\delta \}, \delta=\pm1
\]
and. 

We next define a \emph{De Sitter sphere bundle.} Consider the Clifford span
 \[
 \Sigma:= \text{span}\{P_1,\ldots,P_n\}\subset\text{Sym}(\R_s^{2k})
 \]
and define $\Sigma(\pm1):=\{P\in\Sigma \;:\; \langle P,P \rangle=\pm 1 \}$. Let $E_\pm(P)=\ker(P\mp 1)$ be the eigenspaces of $P\in\Sigma$ and for $P\in\Sigma(+1)$ set $S(P):=E_+(P)\cap\S^{2k-1}_s$. We define
\[
\Sigma^\ast:=\{ P\in\Sigma(+1) \;:\; S(P)\neq\emptyset \} \ \text{ and }\ 
\Gamma:= \{ (P,z) \;:\; P\in\Sigma^\ast, z\in S(P) \}.
\]
Then the map $\pi:\Gamma\rightarrow\Sigma^\ast$ given by $\pi(P,z)=P$ is a De Sitter sphere bundle over the \emph{Clifford sphere } $\Sigma^\ast$.

If $M$ is a nondegenerate hypersurface of $\S_s^m$, let $\xi$ be a normal unitary field on $M$ and define $\delta:=\langle \xi,\xi \rangle = Ric(\xi,\xi)\in\{\pm1\}$. When $\delta=1$, we will say that $M$ is of elliptic type and if $\delta=-1$, we will say that $M$ is of hyperbolic type. The main topological and geometric characteristics of the isoparametric hypersurfaces and focal submanifolds for the Clifford examples are resumed in the following theorem. We refer the interested reader to the article \cite{Ha} or the Ph.D. thesis \cite{Ha1} for the proof.

\begin{theorem}
\begin{enumerate}
\item The focal varieties are given by
\[
M_-=\{z\in\S^{2k-1}_s\;:\; z=Pz\text{ for some }P\in\Sigma\}
\] and 
\[
M_+=\{z\in\S^{2k-1}_s\;:\; \langle P_jz,z \rangle = 0, j=1,\ldots,n\}.
\]
\item $M_-$ is a nondegenerate submanifold of $\S_s^{2k-1}$ of codimension $m_2+1$ and diffeomorphic to the total space $\Gamma$ of the De Sitter sphere bundle. The focal variety $M_+$ is a nondegenerate submanifold of codimension $m_1+1$ in $\S^{2k-1}_s$ and trivial normal bundle.
\item Let $c\in\R\smallsetminus\{-1,+1\}$. Then $M_c$ is diffeomorphic to $TB_-(+1)\approx TB_+(+1)$ when $-1<c<1$, to $TB_-(-1)$ when $c<-1$ and to $TB_+(-1)$ when $c>1$. 
\item The hypersurfaces $M_c$, $c\neq\pm 1$ have four distinct principal curvatures $k_1,\ldots k_4$ with multiplicities $(m_1,m_2,m_1,m_2)$. When $-1<c<1$, the hypersurface is of elliptic type and all the principal curvatures are real, while if $\vert c\vert >1$, then the hypersurface is of hyperbolic type, $k_1,k_2$ are real and $k_2,k_4$ are complex.
\end{enumerate}
\end{theorem}

\section{An existence and uniqueness result for singular second order ODE's}\label{App:ExistenceUniquenessSingular}

We next give here an existence and uniqueness result for a large class of singular nonlinear ODE's.

Let $f:\R\rightarrow\R$ be locally Lipschitz continuous, let $t_0\in\R$ and $T>0$. Let $\rho,q\in C^{1}([t_0-T,t_0+T])$ and consider the problem
\begin{equation}\label{Eq:ExistenceUniquessProblem}
\left\{\begin{tabular}{cc}
$(\rho v')' = qf(v)$ & in $[t_0-T,t_0+T]$\\
$v(t_0)=d_0 \ \  v'(t_0)=c_0$. &
\end{tabular}\right.
\end{equation}

\begin{proposition}\label{Prop:ExistenceUniqueness}
Suppose $\rho,q\geq0$ in $[t_0-T,t_0+T]$, then
\begin{itemize}
\item If $\rho(t_0)\neq0$, for any $d_0,c_0\in \R$ there exists $0<\delta\leq T$ such that the problem \eqref{Eq:ExistenceUniquessProblem} has a unique solution defined in $[t_0-\delta,t_0+\delta]$.
\item If $\rho(t_0)=0$, $\rho(t)\neq0$ for $t\neq t_0$, $\lim_{t\rightarrow t_0}\frac{q(t)}{\rho'(t)}\in\R$ and if there exists $N>0$ such that
\[
\frac{1}{\rho(t)}\int_{t_0}^tq(s)ds\leq N,\ \text{ for every }t\in[t_0-T,t_0+T],
\]
then, for any $d_0\in\R$ and for $c_0:=\frac{q(t_0)}{\rho'(t_0)}f(d_0)$, problem \eqref{Eq:ExistenceUniquessProblem} has a unique solution defined in $[t_0-\delta,t_0+\delta]$ for some $0<\delta\leq T$.
\end{itemize}
\end{proposition}

\begin{proof}
The proof for $\rho(t_0)\neq 0$ is a standard application of the contraction mapping theorem, taking $\delta>0$ in such a way that $\rho>0$ in $[t_0-\delta,t_0+\delta]$. We proceed to the proof of the case $\rho(t_0)=0$.

As $f$ is locally Lipschitz continuous, there exists $C>0$ and $\delta_0>0$ such that if $d_1,d_2\in[d_0-\delta_0,d_0+\delta_0]$, then $\vert f(d_1)-f(d_2)\vert \leq C\vert d_1-d_2 \vert$. As $f$ is continuous in $[t_0-\delta,t_0+\delta]$, there exists $M>0$ such that $\vert f(d)\vert\leq M$ for every $d\in[t_0-\delta_0,t_0+\delta_0]$. Define $\delta:=\min\{\frac{\delta_0}{2MN},\frac{1}{4CN},T\}$ and consider 
\[
X:=\{v\in C^0([t_0-\delta,t_0+\delta]) \;:\; \Vert v - d_0\Vert_\infty \leq 2MN\delta \}
\]  
Thus, is $(X,\Vert\cdot\Vert_\infty)$ is a complete metric space, for it is just the closed ball with center in $v\equiv d_0$ and radius $2MN\delta$ in $(C^0([t_0-\delta,t_0+\delta],\Vert\cdot\Vert_\infty)$. Observe that $v\in X$ implies that $v(t)\in[d_0-\delta_0,d_0+\delta_0]$ for every $t\in[t_0-\delta,t_0+\delta]$.
Define $S:X\rightarrow X$ given by
\[
S(v)(t):= d_0 + \int_{t_0}^t\frac{1}{\rho(s)}\left\{ \int_{t_0}^s q(r)f(v(r)) dr \right\}ds.
\]
By the uniform bound for $P(t):=\frac{1}{\rho(t)}\left\{ \int_{t_0}^t q(r)f(v(r)) dr \right\}$ in $[t_0-T,t_0+T]$ and the definition of $\delta$, a well known argument yields that $S$ is a well defined contraction. Hence, the contraction mapping theorem implies the existence of a unique $v\in X$ such that $S(v)=v$,  giving a unique  solution to the problem \eqref{Eq:ExistenceUniquessProblem} with initial condition $v(t_0)=d_0$. To see the initial condition on the derivative, taking the derivative in $S(v)(t)=v(t)$ we get that
\[
v'(t)=\frac{\int_{t_0}^tq(s) f(v(s)) ds}{\rho(t)}.
\] 
By L'H\^{o}pital's rule and the existence of the limit $\lim_{t\rightarrow t_0}\frac{q(t)}{\rho'(t)}$ we conclude that
\[
v'(t_0)=\lim_{t\rightarrow t_0}\frac{\int_{t_0}^tq(s) f(v(s)) ds}{\rho(t)}=\lim_{t\rightarrow t_0}\frac{q(t) f(v(t)) }{\rho'(t)}=\frac{q(t_0)  }{\rho'(t_0)}f(d_0).
\]
Therefore $v$ satisfies the problem \eqref{Eq:ExistenceUniquessProblem} with the natural boundary condition for the derivative.
\end{proof}

\end{document}